\documentclass[11pt]{amsart}
\usepackage{stmaryrd}
\usepackage{amssymb}
\usepackage{csvsimple}
\usepackage{siunitx}
\usepackage{hyperref}
\usepackage{url}
\usepackage{graphicx,color}
\usepackage{lineno}           
\usepackage[style=trad-abbrv,url=false,doi=true,
  eprint=true,isbn=false,backend=biber]{biblatex}

\graphicspath{{pics/},{figs/}}
\def\R{\mathbb{R}}

\def\T{\mathbb{T}}
\def\Rinf{\R\cup \{+\infty\}}

\def\cC{\mathcal{C}}

\def\cI{\mathcal{I}}
\def\cJ{\mathcal{J}}

\def\cL{\mathcal{L}}

\def\cO{\mathcal{O}}

\def\cS{\mathcal{S}}
\def\cT{\mathcal{T}}

\def\a{\alpha}
\def\b{\beta}
\def\g{\gamma}
\def\d{\delta}

\def\l{\lambda}
\def\s{\sigma}
\def\p{\partial}

\def\veps{\varepsilon}
\def\vrho{\varrho}
\def\vphi{\varphi}
\def\O{\Omega}
\def\G{\Gamma}
\def\GD{{\Gamma_D}}
\def\GN{{\Gamma_N}}

\def\wto{\rightharpoonup}

\def\ou{\overline{u}}

\def\tu{\widetilde{u}}

\newcommand{\dv}[1]{\,{\mathrm d}#1}

\newcommand{\trnorm}[1]{|\negthinspace|\negthinspace|#1|\negthinspace|\negthinspace|}
\newcommand{\wcheck}[1]{#1\hspace{-.8ex}\mbox{\huge {\lower.45ex \hbox{$\textstyle \check{}$}}} \hspace{.5ex}}

\newcommand{\jump}[1]{\llbracket#1\rrbracket}   
\DeclareMathOperator{\id}{id}

\DeclareMathOperator{\diver}{div}
\def\nablah{\nabla_{\!h}}

\DeclareMathOperator{\diam}{diam}

\let\oldmarginpar\marginpar
\renewcommand\marginpar[1]{
  \oldmarginpar[\raggedleft\footnotesize #1]
  {\raggedright\footnotesize #1}}

\newtheorem{definition}{Definition}

\newtheorem{proposition}[definition]{Proposition}
\newtheorem{theorem}[definition]{Theorem}
\newtheorem{corollary}[definition]{Corollary}
\newtheorem{remark}[definition]{Remark}
\newtheorem{remarks}[definition]{Remarks}
\newtheorem{example}[definition]{Example}

\numberwithin{definition}{section}
\definecolor{modmag}{RGB}{179,0,229}

\def\RT{{\mathcal{R}T}}

\def\ou{\overline{u}}

\def\tchi{\widetilde{\chi}}

\def\nablah{\nabla_{\!h}}

\def\tz{{\widetilde{z}}}

\def\DD{\mathrm{D}}

\begin{document}
\title[DG methods for nonsmooth problems]{Error estimates for
a class of discontinuous Galerkin methods for nonsmooth problems 
via convex duality relations}
\author{S\"oren Bartels}
\address{Abteilung f\"ur Angewandte Mathematik,  
Albert-Ludwigs-Universit\"at Freiburg, Hermann-Herder-Str.~10, 
79104 Freiburg i.~Br., Germany}
\email{bartels@mathematik.uni-freiburg.de}
\date{\today}
\renewcommand{\subjclassname}{
\textup{2010} Mathematics Subject Classification}
\begin{abstract}
We devise and analyze a class of interior penalty discontinuous
Galerkin methods for nonlinear and nonsmooth variational problems.
Discrete duality relations are derived that lead to optimal
error estimates in the case of total-variation regularized 
minimization or obstacle problems. The analysis provides 
explicit estimates that precisely determine the role of stabilization 
parameters. Numerical experiments suppport the optimality
of the estimates. 
\end{abstract}

\maketitle

\section{Introduction}
\subsubsection*{Total-variation minimization}
As a particular example of a nonsmooth convex variational problem we  
consider the total variation regularized optimization problem that
determines a function $u\in BV(\O)\cap L^2(\O)$ via a minimization of
\[
I(u) = |Du|(\O) + \frac{\a}{2}\|u-g\|^2,
\]
where $|Du|(\O)$ is the total variation of $u\in L^2(\O)$, which coincides
with $\|\nabla u\|_{L^1(\O)}$ if $u\in W^{1,1}(\O)$, while
$\a>0$ and $g\in L^2(\O)$ are given data, cf.~\cite{AmFuPa00-book,AtBuMi06-book,CCCNP10} 
for analytical features and numerical methods.  
Since discontinuous solutions
are expected and since continuous methods are known to provide suboptimal
results~\cite{BaNoSa15,Bart20-pre}, it is attractive to discretize the minimization problem
by a discontinuous finite element method, e.g., via determining an
elementwise affine, possibly discontinuous function $u_h \in \cS^{1,dg}(\cT_h)$
on a triangulation $\cT_h$ as a minimizer of the functional
\[
I_h(u_h) = \int_\O |\nablah u_h| \dv{x} 
+ \frac{1}{r} \int_{\cS_h\setminus \p\O}  h_\cS^{-\g r} |\jump{u_h}_h|^r \dv{s}
+ \frac{\a}{2} \| \Pi_h(u_h - g)\|^2. 
\]
Here, $\nablah$ denotes the elementwise application of the gradient,
$\cS_h$ stands for the union of element sides in $\cT_h$, the operator
$\Pi_h$ is the projection onto piecewise
constant functions or vector fields on $\cT_h$, the function
$h_\cS:\cS_h \to \R_{>0}$ is a mesh-size function, and $\jump{\cdot}$ and 
$\jump{\cdot}_h$ denote the jump and the mean of a jump
of a piecewise polynomial function. Since $u_h$ is piecewise affine
we have that $\Pi_h u_h|_T = u_h(x_T)$ for every element $T\in \cT_h$
with midpoint (barycenter) $x_T$ and 
\[
\jump{u_h}_h|_S = \jump{u_h}(x_S) =  \lim_{\veps \to 0} u_h(x+\veps n_S) - u_h(x_S - \veps n_S)
\]
for every side $S\in\cS_h$ with midpoint (barycenter) $x_S$ and unit normal 
$n_S$. The second term in the discrete energy functional $I_h$ thus 
penalizes averages of jumps across interelement sides. The use of the mean
has the alternative interpretation of using quadrature which makes the scheme
practical. The choice of the parameters
$r$ and $\g$ is crucial for an accurate approximation of the exact solution $u$.
We remark that our approach is motivated and inspired by recent results 
in~\cite{ChaPoc19-pre,Bart20-pre} 
on discretizations of nonsmooth problems using Crouzeix--Raviart elements. 

To quantify the accuracy we define a discrete dual problem. We show
that a naturally associated maximization problem is defined for a discrete vector field
$z_h\in \RT^0_{\!N}(\cT_h)$ by the functional
\[\begin{split}
D_h(z_h) =& - I_{K_1(0)} (\Pi_{0,h} z_h)  
- \frac{1}{r'} \int_{\cS_h\setminus \GN} h_\cS^{\g r'} |\{z_h\cdot n_\cS\}|^{r'} \dv{s} \\
&- \frac{1}{2\a} \|\diver z_h + \a g_h\|^2 + \frac{\a}{2} \|g_h\|^2,
\end{split}\]
where $r' = r/(r-1)$. 
The Raviart--Thomas finite element space $\RT^0_{\!N}(\cT_h)$ consists of 
certain elementwise affine vector fields whose weak divergence is a function
and whose normal component vanishes on $\GN = \p\O$. In
particular, the normal components $z_h\cdot n_\cS$ are continuous and
constant on element sides, so that their averages $\{z_h\cdot n_\cS\}$ 
coincide with the values $z_h\cdot n_\cS$ on every side. It turns out that the 
penalty terms in the discrete primal problem $I_h$  
are related to stabilizing terms on element sides in the discrete dual 
functional~$D_h$.
The indicator functional $I_{K_1(0)}$ enforces 
the length of the vector field $z_h$ to be bounded by~1 at element 
midpoints. Thus, in the discrete duality relation, jumps of functions in the
primal problem lead to averages in the dual formulation. On 
the continuous level the dual formulation consists in maximizing
the functional 
\[
D(z) = -I_{K_1(0)}(z) - \frac{1}{2\a} \|\diver z + \a g\|^2 + \frac{\a}{2} \|g\|^2
\]
in the set of vector fields $z\in W^2_N(\diver;\O)$ whose distributional divergence
belongs to $L^2(\O)$ and whose normal component vanishes on $\GN= \p\O$. Strong
duality applies, i.e., we have $I(u)=D(z)$ for solutions~$u$ and~$z$, which is
a well-posedness property of the variational problem.

An error estimate follows from coercivity properties of $I_h$ and
the crucial discrete duality relation $I_h(u_h) \ge D_h(z_h)$. More precisely, 
with appropriately 
defined quasi-interpolants $\cI_h u$, that is continuous at midpoints of
element sides, and $\cJ_h z$, for a sufficiently regular solution~$z$
of the continuous dual problem, we have
\[
\frac{\a}{2}\|\Pi_h (u_h - \cI_h u)\|^2 \le I_h(\cI_h u) - I_h(u_h) 
\le I_h(\cI_h u) - D_h(\cJ_h z).
\]
The quasi-interpolants
are defined in such a way that we have
\[
\nablah \cI_h u = \Pi_h \nabla u, \quad 
\diver \cJ_h z = \Pi_h \diver z.
\]
These relations allow us apply Jensen's inequality, which in the
present context has the interpretation of a total-variation diminishing 
interpolant, and thereby leads to the discrete error estimate
\[
\frac{\a}{2}\|\Pi_h (u_h - \cI_h u)\| \le c h^{1/2},
\]
provided that $z\in W^{1,\infty}(\O;\R^d)$, $u\in L^\infty(\O)$,
and $\g r' \ge 2$ or $\g \ge 0$ if $r=1$.  
The estimate implies the error bound
\[
\|u- \Pi_h u_h\| \le c h^{1/2},
\]
where $\Pi_h u_h$ can be replaced by $u_h$ provided that the
sequence $(u_h)_{h>0}$ is uniformly bounded in $L^\infty(\O)$. 
The convergence rate $\cO(h^{1/2})$ coincides with the rate for  
Crouzeix--Raviart finite element methods, cf.~\cite{ChaPoc19-pre,Bart20-pre},
and is quasi-optimal for the approximation of a generic function in 
$BV(\O)\cap L^\infty(\O)$. The optimal rate can in general not
be obtained with continuous finite element methods \cite{Bart12,BaNoSa15}.
Note that since $\g r= 0$ is allowed the approach is not a pure penalty 
method. On the other hand, it does not suffer from locking effects for
large values of $\g r$ owing to the use of quadrature in the 
jump contributions. We refer the reader to~\cite{CaiCha20-pre} for
discretizations using finite difference methods. 

\subsubsection*{General error analysis}
The analysis summarized for the numerical approximation of the total
variation-regularized problem by discontinuous methods can be generalized
in several ways. For this, we consider a convex mimimization
problem defined via (a suitable extension of) the functional
\[
I(u) = \int_\O \phi(\nabla u) + \psi(x,u) \dv{x}
\]
on a Sobolev space $W^{1,p}_D(\O)$ of functions with vanishing traces on
$\GD\subset \p\O$. The dual formulation is given by a maximization of the
functional 
\[
D(z) = -\int_\O \phi^*(z) + \psi^*(x,\diver z) \dv{x},
\]
in a space $W^q_N(\diver;\O)$ of vector fields in $L^q(\O;\R^d)$ 
whose normal components vanish on $\GN = \p\O\setminus \GD$ and whose 
distributional divergence belongs to $L^q(\O)$, where $q=p'$ is the
conjugate exponent of~$p$.

A class of discontinuous Galerkin 
discretizations of the primal problem is given by the functionals 
\[\begin{split}
I_h(u_h) =  \int_\O & \phi(\nablah u_h) + \psi_h(x,\Pi_h u_h) \dv{x}   \\
& + \frac1r \int_{\cS_h\setminus \GN} \a_\cS^{-r} |\jump{u_h}_h|^r \dv{s}  
+ \frac1s \int_{\cS_h\setminus \GD} \b_\cS^s |\{u_h\}_h|^s \dv{s}
\end{split}\]
with suitable exponents $r,s\ge 1$ and weights 
$\a_\cS,\b_\cS:\cS_h \to \R_{\ge 0}$ on the element sides. The 
quantities $\{u_h\}$ are on every side the average of traces
of $u_h$ from adjacent elements, its mean $\{u_h\}_h$ coincides for
piecewise affine functions with the evaluation at the midpoint of $S$, i.e., 
\[
\{u_h\}_h = \{u_h\}(x_S) 
= \lim_{\veps \to 0} \frac12 \big(u_h(x+\veps n_S) + u_h(x_S - \veps n_S)\big).
\]
We show that a discrete duality argument leads to the discrete dual
functional 
\[\begin{split}
D_h(z_h) =& - \int_\O \phi^*(\Pi_{0,h} z_h) + \psi_h^*(x,\diver_h z_h) \dv{x} \\
& - \frac{1}{r'} \int_{\cS_h \setminus \GN} \a_\cS^{r'} |\{z_h\cdot n\}|^{r'} \dv{s}  
- \frac{1}{s'} \int_{\cS_h \setminus \GD} \b_\cS^{-s'} |\jump{z_h\cdot n}|^{s'} \dv{s}.
\end{split}\]
This functional is defined on a broken Raviart--Thomas finite element space
$\RT^{0,dg}(\cT_h)$. Here, jumps of the normal component of $z_h$ across
element sides are penalized which corresponds to the presence of
the averages $\{u_h\}$ on element sides in the discrete primal problem. 
The duality of jumps and averages on element sides is a result of an elementwise 
integration by parts and the elementary formula for inner sides 
\[
\jump{u_h z_h \cdot n_S} 
= \jump{u_h}\{z_h \cdot n_S\} + \{u_h\} \jump{z_h\cdot n_S},
\]
that relates jumps of products to products of jumps and averages. Together
with Fenchel's inequality it leads to the important discrete duality 
relation
\[
I_h(u_h) \ge D_h(z_h)
\]
for solutions $u_h$ and $z_h$ of the discrete primal and dual problems,
respectively.
As above, this inequality is important for an error analysis. In particular,
it provides full control on nonconformity errors which otherwise require
the use of a Strang lemma or suitable reconstruction operatos, 
cf., e.g.,~\cite{Ciar78-book,BreSco08-book}. Here, 
these error contributions are entirely controlled via structure-preserving
features of the discretizations. 

\subsubsection*{Nonlinear Dirichet and obstacle problems}
For a class of nonlinear Dirichlet problems with linear low order terms 
given by 
\[
\psi(x,s) = -f(x) s
\]
we obtain with an appropriate coercivity functional $\s$ the general and
constant-free discrete error estimate 
\[\begin{split}
\int_\O \s (\nablah u_h,\nablah & \cI_h u) \dv{x} 
 \le \int_\O \big(D\phi^*(z)-D\phi^*(\Pi_h \cJ_h z)\big) \cdot (z-\Pi_h \cJ_h z)\dv{x} \\
&  + \frac{1}{r'} \|\a_\cS \{\cJ_h z\cdot n_\cS\}\|_{L^{r'}(\cS_h \setminus \GN)}^{r'}
+ \frac{1}{s'} \|\b_\cS \{\cI_h u\}_h\|_{L^{s'}(\cS_h \setminus \GD)}^{s'},
\end{split}\]
with quasi-interpolants $\cI_h u$ and $\cJ_h z$ of sufficiently regular
primal and dual solutions $u$ and $z$, respectively. The concepts also
apply to obstacle problems, for which the low order term is given
by
\[
\psi(x,s) = -f(x) s + I_{\R_{\ge 0}}(s),
\]
with the indicator function $I_{\R_{\ge 0}}$ that enforces the 
solution $u$ of the primal problem to be nonnegative. In the case of
a quadratic functional $\phi(v) = |v|^2/2$ and with $r=s=2$ we obtain for 
regular solutions $u\in W^{1,2}_D(\O)\cap W^{2,2}(\O)$ the 
constant-free discrete error estimate
\[\begin{split}
\frac12 \|\nablah (u_h-\cI_h u)&\|^2  \le \frac12 \|z-\Pi_h \cJ_h z\|^2 \\
& \quad + \|f+\Delta u\| \big(\|u-\cI_h u\| + h \|\nablah (u-\cI_h u)\|\big)  \\
& + \frac{1}{2} \|\a_\cS \{\cJ_h z\cdot n_\cS\}\|_{L^2(\cS_h \setminus \GN)}^2
+ \frac{1}{2} \|\b_\cS \{\cI_h u\}_h\|_{L^2(\cS_h \setminus \GD)}^2.
\end{split}\]
The right-hand side is of quadratic order if 
\[
\big\|\a_\cS h_\cS^{-3/2}\|_{L^\infty(\cS_h)}
+ \big\|\b_\cS  h_\cS^{-3/2}\|_{L^\infty(\cS_h)}  \le c_\cT,
\]
i.e., if $\a_\cS =\b_\cS = \cO(h_\cS^{3/2})$.
Note that $\a_\cS>0$ is needed for well-posedness
of the discrete problem while if $\b_\cS=0$ then the contribution
to $D_h$ involving the jumps $\jump{z_h\cdot n_\cS}$ becomes an indicator
functional and the space $\RT^{0,dg}(\cT_h)$ has to be replaced by 
$\RT^0_N(\cT_h)$. 

\subsubsection*{Relations to other methods}
In contrast to established discontinuous Galerkin methods for
elliptic problems as in~\cite{Arno82,ABCM01,CCPS00,diPErn12-book},
which are typically derived 
from strong formulations of partial differential equations, we obtain
here more restrictive conditions on the penalty parameters in the
case of differentiable elliptic equations to obtain quasi-optimal error
estimates. In these cases the simple
interior penalty approach realized by the discrete minimization 
problems~$I_h$ is inconsistent with the weak formulations of the 
corresponding partial differential equations. While the variational
approach for penalty based discontinuous Galerkin methods for variational
problems considered in~\cite{BufOrt09} shows that convergence 
is guaranteed under mild conditions, our numerical experiments 
confirm that they are not sufficient to obtain optimal convergence
rates. For the nondifferentiable
total-variation problem our discretizations do not require 
penalizations and our approach is conistent with a natural discretization
of the total-variation norm. Generally, the error analysis used
here only uses optimality conditions in terms of first order
system and subdifferentials. 
Another advantage of our error analysis
based on duality arguments is that it provides explicit estimates
that do not require absorbing terms and hence precisely determine the
role of the parameters involved in the discontinuous Galerkin
discretization. Additionally, it leads to a realistic and optimal
regularity condition in terms of the dual solution. Throughout this
article we use simplicial partitions which is important for the 
error analysis. Our duality arguments transfer verbatimly to 
general classes of polyhedral partitions.

\subsubsection*{Outline}
The outline of this article is as follows. In Section~\ref{sec:fe_spaces}
we introduce notation and define appropriate finite element spaces.
A discrete duality theory is provided in Section~\ref{sec:discr_conj}.
The application to nonlinear Dirichlet problems, total variation
minimization, and elliptic obstacle problems is discussed in the
subsequent Sections~\ref{sec:nonlin_diri}-\ref{sec:obstacle}.
Numerical experiments are presented in Section~\ref{sec:num_ex}.


\section{Notation and finite element spaces}\label{sec:fe_spaces}
For a sequence of regular triangulations $(\cT_h)_{h>0}$, where $h>0$ 
refers to a maximal mesh-size that tends to zero, the set of
elementwise polynomial functions or vector fields of maximal 
polynomial degree $k \ge 0$ is defined by
\[
\cL^k(\cT_h)^\ell 
= \big\{v_h \in L^1(\O;\R^\ell): 
v_h|_T \in P_k(T)^\ell \mbox{ for all } T\in \cT_h\big\}.
\]
We let $\Pi_h : L^1(\O;\R^\ell) \to \cL^0(\cT_h)^\ell$ denote the
$L^2$ projection onto elementwise constant functions or vector
fields and note that $\Pi_h$ is self-adjoint, i.e., 
\[
\int_\O \Pi_h f g \dv{x} = \int_\O f \Pi_h g \dv{x} 
\] 
for all $f,g\in L^1(\O)$.
We let $\cS_h$ denote the set of sides of elements and 
define the mesh-size function $h_\cS|_S = h_S = \diam(S)$ for all 
sides $S\in \cS_h$. We let $n_\cS:\cS_h \to \R^d$ denote a unit 
vector field given for every side $S\in \cS_h$ by
\[
n_\cS|_S = n_S
\]
for a fixed unit normal $n_S$ on $S$ which is assumed to coincide 
with the outer unit normal if $S\subset \p\O$. The jump and average on a side $S$ 
of a function $v_h\in \cL^k(\cT_h)^\ell$ are for $x\in S$ defined for inner
sides via
\[\begin{split}
\jump{v_h}(x) &= \lim_{\veps \to 0} \big(v_h(x-\veps n_S)- v_h(x+\veps n_S)\big), \\
\{v_h\}(x) &=  \lim_{\veps \to 0} \frac12 \big(v_h(x-\veps n_S)+ v_h(x+\veps n_S)\big).
\end{split}\]
For $S\subset \p\O$ we set 
\[
\jump{v_h} = \{v_h\} = v_h.
\]
The integral means of jumps and averages are denoted by
\[
\jump{v_h}_h = |S|^{-1} \int_S \jump{v_h} \dv{s}, \quad 
\{v_h\}_h = |S|^{-1} \int_S \{v_h\} \dv{s},
\]
which in case of elementwise affine functions coincides with the evaluation
at the midpoint $x_S$ for every $S\in \cS_h$. We define the space of
discontinuous, piecewise linear functions via 
\[
\cS^{1,dg}(\cT_h) = \cL^1(\cT_h).
\]
A space of discontinuous vector fields is given by
\[
\RT^{0,dg}(\cT_h) = \cL^0(\cT_h)^d + (\id - x_\cT) \cL^0(\cT_h),
\]
where $\id$ is the identity and $x_\cT = \Pi_h \id \in \cL^0(\cT_h)^d$ the 
elementwise constant vector field that coincides with the midpoint $x_T$ on every
$T\in \cT_h$. Differential operators on these spaces are defined 
elementwise, indicated by a subscript $h$, i.e., we have
\[
\nablah v_h|_T = \nabla (v_h|_T), \quad \diver_{\!h} z_h|_T = \diver (z_h|_T)
\]
for $v_h\in \cS^{1,dg}(\cT_h)$, $z_h\in \RT^{0,dg}(\cT_h)$ and 
all $T\in \cT_h$. The operators are also applied to weakly differentiable
functions and vector fields in which case they coincide with the weak
gradient and the weak divergence. 
By construction, any vector field $y_h\in \RT^{0,dg}(\cT_h)$
has a piecewise constant normal component $y_h\cdot n_L$ along straight 
lines $L$ with normal $n_L$. Subspaces of elementwise affine functions
and vector fields 
with certain continuity properties on element sides are given by
\[
\cS^{1,cr}_D(\cT_h) = \{v_h \in \cS^{1,dg}(\cT_h): \jump{v_h}_h|_S = 0 
\text{ for all $S\in \cS_h\setminus \GN$}\},
\]
and 
\[
\RT^0_{\!N}(\cT_h) = \{y_h \in \RT^{0,dg}(\cT_h): \jump{y_h \cdot n_S}_h|_S = 0
\text{ for all $S\in \cS_h \setminus \GD$}\},
\]
which coincide with low order Crouzeix--Raviart and Raviart--Thomas finite element
spaces introduced in~\cite{CroRav73,RavTho77}.
These spaces provide quasi-interpolation operators 
\[
\cI_h: W^{1,p}_D(\O)\to \cS^{1,cr}_D(\cT_h), \quad
\cJ_h: W^q_N(\diver;\O) \to \RT^0_{\!N}(\cT_h),
\]
with the projection properties 
\[
\nablah \cI_h v = \Pi_h \nabla v, \quad \diver \cJ_h y = \Pi_h \diver y,
\]
and the interpolation estimates 
\[\begin{split}
\|v-\cI_h v \|_{L^p(\O)} &\le c_{\cI,1} h \|\nabla v\|_{L^p(\O)}, \\
\|v- \cI_h v\|_{L^p(\O)} + h \|\nablah (v-\cI_h v)\|_{L^p(\O)} &\le c_{\cI,2} h^2 \|D^2 v\|_{L^p(\O)},
\end{split}\]
for $v\in W^{1,p}_D(\O)$ with $1\le p\le \infty$, and 
\[
\|y- \cJ_h y\|_{L^q(\O)} \le c_\cJ h \|\nabla y\|_{L^q(\O)}
\]
for $y\in W^q_N(\diver;\O)$ with $1\le q \le \infty$. We always assume that
$h$ is sufficiently small so that we have $\|\cI_h v\|_{L^p(\O)} \le \sqrt{2} \|v\|_{W^{1,p}(\O)}$ 
and $\|\cJ_h y\|_{L^q(\O)} \le \sqrt{2} \|y\|_{W^{1,q}(\O)}$. We refer
the reader to~\cite{BoBrFo13-book,BreSco08-book,Bart16-book} for details. Elementary 
calculations lead to the identities
\[
\jump{v_h y_h \cdot n_S} = 
\begin{cases}
\jump{v_h} \{y_h\cdot n_S\} + \{v_h\} \jump{y_h \cdot n_S} & \text{if } S\not \subset \p\O, \\
\jump{v_h} \{y_h\cdot n_S\} & \text{if } S\subset \GD, \\ 
\{v_h\} \jump{y_h\cdot n_S} & \text{if } S\subset \GN.
\end{cases}
\]
By carrying out an elementwise integration by parts we thus find
that for $v_h \in \cS^{1,dg}(\cT_h)$ and $y_h\in \RT^{0,dg}(\cT_h)$ 
we have
\begin{equation}\label{eq:int_parts_dg}
\begin{split}
\int_\O v_h & \diver y_h \dv{x} + \int_\O \nablah v_h \cdot y_h \dv{x} \\
& = \int_{\cS_h \setminus \GN} \jump{v_h}_h \{y_h\cdot n_S\} \dv{s}
+ \int_{\cS_h \setminus \GD} \{v_h\}_h \jump{y_h\cdot n_S} \dv{s}.
\end{split}
\end{equation}
If  $v_h \in \cS^{1,cr}_D(\cT_h)$ and 
$y_h\in \RT^0_{\!N}(\cT_h)$ then the terms on the right-hand side are
equal to zero. We furthermore note that if an elementwise constant vector
field $y_h\in \cL^0(\cT_h)^d$ satisfies
\[
\int_\O y_h \cdot \nablah v_h \dv{x} = 0
\]
for all $v_h\in \cS^{1,cr}_D(\cT_h)$ then it belongs to $\RT^0_{\! N}(\cT_h)$.
This fact follows from an elementwise integration by parts with $v_h=\vphi_S$ for the
Crouzeix--Raviart basis functions $\vphi_S$ associated with sides $S\in \cS_h\setminus \GD$.
To bound functionals defined by integrals on the skeleton 
$\cS_h$ we use the discrete trace inequality 
\begin{equation}\label{eq:trace_est}
\|h_\cS \psi_h\|_{L^s(\cS_h)}^s \le c_\cT \|\psi_h\|_{L^s(\O)}^s,
\end{equation}
for a piecewise linear function $\psi_h \in \cL^1(\cT_h)$, $s\ge 1$,
and a constant $c_\cT$ that depends on the geometry of $\cT_h$. 


\section{Discrete conjugation}\label{sec:discr_conj}

We collect the jump and average terms needed for the discontinuous Galerkin
discretizations in functionals $J_h$ and $K_h$. The results 
of this section apply to general classes of regular polyhedral partitions.

\begin{definition}[Jumps and averages]
Let $r,s\ge 1$ and let $\a_\cS,\b_\cS: \cS_h \to \R_{\ge 0}$ be piecewise
constant. For $u_h\in \cS^{1,dg}(\cT_h)$ and $z_h\in \RT^{0,dg}(\cT_h)$ define 
\[\begin{split}
J_h(u_h)  &= \frac1r \|\a_\cS^{-1} \jump{u_h}_h\|_{L^r(\cS_h \setminus \GN)}^r  
+ \frac1s \|\b_\cS \{u_h\}_h\|_{L^s(\cS_h \setminus \GD)}^s, \\
K_h(z_h) & = \frac{1}{r'} \|\a_\cS \{z_h\cdot n_\cS\}\|_{L^{r'}(\cS_h \setminus \GN)}^{r'}
+  \frac{1}{s'} \| \b_\cS^{-1} \jump{z_h\cdot n_\cS}\|_{L^{s'}(\cS_h \setminus \GD)}^{s'} ,
\end{split}\]
where we require $\jump{u_h}_h = 0$ if $\a_\cS=0$ and 
$\jump{z_h\cdot n_\cS}=0$ if $\b_\cS=0$. For $r=1$ or $s=1$ the
functionals $(1/r')\|\cdot\|_{L^{r'}}^{r'}$ or $(1/s')\|\cdot\|_{L^{s'}}^{s'}$ 
are interpreted as indicator functionals $I_{K_1(0)}$ of the closed  unit ball $K_1(0)$.
\end{definition}

To show that the functionals $J_h$ and $K_h$ are in discrete duality 
we let $\phi^*$ and $\psi^*$ be the convex conjugates of the convex functions 
$\phi$ and $\psi$, i.e., 
\[
\phi^*(y) = \sup_{v\in \R^n} y\cdot v- \phi(v), \quad
\psi^*(x,s) = \sup_{t\in \R} s \,t - \psi(x,t).
\]
For simple power functionals the conjugate is determined by the conjugate
exponent, i.e., for a factor $c\ge 0$ and an exponent $\s\ge 1$ we have
\[
g(v) = \frac{1}{\s} c^\s |v|^\s \quad \Longleftrightarrow \quad 
g^*(w) = 
\begin{cases}
\frac{1}{\s'} c^{-\s'} |w|^{\s'} & \mbox{for } \s > 1, \\
I_{K_1(0)}(c^{-1} w) & \mbox{for } \s = 1,
\end{cases}
\]
where $\s' = \s/(\s-1)$ and $I_{K_1(0)}$ is the indicator functional of 
the closed unit ball around~0. The definition of $g^*$ leads to 
Fenchel's inequality
\begin{equation}\label{eq:fenchel}
v\cdot w \le g(v) + g^*(w),
\end{equation}
where equality holds if and only if $w= Dg(v)$ or equivalently
$v= D g^*(w)$. 

\begin{proposition}[Discrete conjugation]\label{prop:discr_conv_conj}
For $u_h\in \cS^{1,dg}(\cT_h)$ and $z_h \in \RT^{0,dg}(\cT_h)$ define
\[\begin{split}
V_h(u_h) &= \int_\O \phi(\nablah u_h) \dv{x} + J_h(u_h), \\
W_h(z_h) &= - \int_\O \phi^*(\Pi_{0,h} z_h) \dv{x} - K_h(z_h).
\end{split}\]
Given any $\ou_h\in \cL^0(\cT_h)$ we have that
\[\begin{split}
\inf\big\{V_h(u_h): & \, u_h \in \cS^{1,dg}(\cT_h), \, \Pi_h u_h = \ou_h\big\} \\
& \ge \sup\big\{ W_h(z_h) - (\ou_h, \diver_h z_h) : z_h \in \RT^{0,dg}(\cT_h) \big\}.
\end{split}\]
\end{proposition}

\begin{proof}
For arbitrary $z_h \in \RT^{0,dg}(\cT_h)$ and $u_h \in \cS^{1,dg}(\cT_h)$ 
with $\Pi_h u_h = \ou_h$ we use the integration-by-parts 
formula~\eqref{eq:int_parts_dg} to verify that 
\[\begin{split}
W_h(z_h) - (\ou_h,&\diver_h z_h) 
= - \int_\O \phi^*(\Pi_{0,h} z_h) \dv{x} + \int_\O \nablah u_h \cdot \Pi_h z_h \dv{x} \\
&\quad - \frac{1}{r'} \int_{\cS_h \setminus \GN} \a_\cS^{r'} |\{z_h\cdot n_\cS\}|^{r'} \dv{s}  
- \int_{\cS_h \setminus \GN} \jump{u_h}_h \{z_h \cdot n_\cS\} \dv{s} \\
&\quad - \frac{1}{s'} \int_{\cS_h \setminus \GD} \b_\cS^{-s'} |\jump{z_h\cdot n_\cS}|^{s'} \dv{s}
- \int_{\cS_h\setminus \GD} \{u_h\}_h \jump{z_h \cdot n_\cS}\dv{s}.
\end{split}\]
With Fenchel's inequality we deduce that
\[
- \phi^*(\Pi_{0,h} z_h)  + \nablah u_h \cdot \Pi_h z_h \le \phi(\nablah u_h),
\]
and
\[
- \frac{1}{r'}  \a_\cS^{r'} |\{z_h\cdot n_\cS\}|^{r'}  
+  (-\jump{u_h}_h) \{z_h \cdot n_\cS\} 
\le \frac{1}{r}  \a_\cS^{-r} |\jump{u_h}_h|^r , 
\]
as well as 
\[
- \frac{1}{s'} \b_\cS^{-s'} |\jump{z_h\cdot n_\cS}|^{s'} 
+ (-\{u_h\}_h) \jump{z_h \cdot n_\cS}
\le  \frac{1}{s} \b_\cS^s |\{u_h\}_h|^s .
\]
On combining the estimates
and noting that $u_h$ and $z_h$ are arbitrary, we deduce the statement. 
\end{proof}

The discrete convex conjugates lead to a canonical
discrete dual problem.

\begin{theorem}[Discrete duality]\label{thm:discr_duality} 
For $u_h\in \cS^{1,dg}(\cT_h)$ let 
\[
I_h(u_h) = \int_\O \phi(\nablah u_h) + \psi_h(x,\Pi_h u_h) \dv{x} + J_h(u_h),
\]
where $\psi_h:\O\times \R \to \Rinf$ is elementwise constant with respect
to the first argument. Then 
with the discrete dual functional defined for $z_h\in \RT^{0,dg}(\cT_h)$ by
\[
D_h(z_h) = -\int_\O \phi^*(\Pi_h z_h) + \psi_h^*(x,\diver_h z_h) \dv{x} -K_h(z_h)
\]
we have
\[
I _h(u_h) \ge D_h(z_h).
\]
\end{theorem}

\begin{proof}
With the inequality of Proposition~\ref{prop:discr_conv_conj} we have,
using $\ou_h = \Pi_h u_h$ that 
\[\begin{split}
&\int_\O \phi(\nablah u_h) \dv{x} + J_h(u_h) + \int_\O \psi_h(x,\ou_h) \dv{x} \\
&\ge - \int_\O \phi^*(\Pi_{0,h} z_h) \dv{x} - K_h(z_h) - (\ou_h, \diver_h z_h) + \int_\O \psi_h(x,\ou_h) \dv{x}.
\end{split}\]
Fenchel's inequality shows that on every $T\in \cT_h$ we have 
\[
\ou_h \diver_h z_h \le  \psi_h(x,\ou_h) + \psi_h^*(x,\diver_h z_h).
\]
This implies the asserted inequality. 
\end{proof}

A strong duality relation can be established under additional conditions. 
We consider a particular but typical definition of the penalty terms. 
The formula stated in the following proposition 
provides a discrete dual solution via a simple postprocessing
of the discreteprimal solution and generalizes a result from~\cite{Mari85}.

\begin{proposition}[Reconstruction and strong duality]
Assume that $J_h$ and $K_h$ are defined with $r=s=2$ and $\b_\cS=0$
and that $\phi$ and $\psi$ are continuously differentiable.
If $u_h\in \cS^{1,dg}(\cT_h)$ is minimal for $I_h$ then the
vector field
\[
\tz_h = D\phi(\nablah u_h) + d^{-1} D\psi_h(\Pi_h u_h) (\id-x_\cT)
\]
belongs to $\RT^0_{\!N}(\cT_h)$ and is maximal for $D_h$ with
$I_h(u_h) = D_h(\tz_h)$. 
\end{proposition}

\begin{proof}
We note that $u_h$ solves the discrete Euler--Lagrange equations
\[\begin{split}
\int_\O D\phi(\nablah u_h) \cdot \nablah  v_h \dv{x} 
+ \int_\O & D\psi_h(\Pi_h u_h) v_h \dv{x} \\
&= - \int_{\cS_h\setminus \GN} \a_\cS^{-2} \jump{u_h}_h \jump{v_h}_h \dv{s}
\end{split}\]
for all $v_h \in \cS^{1,dg}(\cT_h)$ and the term on the 
right-hand side vanishes if $v_h\in \cS^{1,cr}_D(\cT_h)$. To show that 
$\tz_h\in \RT^0_{\!N}(\cT_h)$ we choose $y_h\in \RT^0_{\!N}(\cT_h)$
with $-\diver y_h = D\psi_h(\Pi_h u_h)$. Then, 
$\tz_h - y_h$ is elementwise constant and we have 
\[
\int_\O (\tz_h -y_h) \cdot \nablah v_h \dv{x}
= \int_\O \big(D\phi(\nablah u_h) -y_h) \cdot \nablah v_h\dv{x}
= 0
\]
for all $v_h\in \cS^{1,cr}_D(\cT_h)$. In particular, we
deduce that $\tz_h - y_h \in \RT^0_{\!N}(\cT_h)$ and therefore
$\tz_h\in \RT^0_{\!N}(\cT_h)$. Given a side $S\in \cS_h\setminus \GD$
with adjacent element $T\in \cT_h$ we let
$\vphi_{S,T} \in \cS^{1,dg}(\cT_h)$ be the function that is supported
on $T$, vanishes in those midpoints of sides that do not belong
to $S$, and satisfies $\vphi_{S,T}(x_S) = 1$. We have $\jump{\vphi_{S,T}}_h =-1$
on $S$ and the discrete Euler--Lagrange equations
and an integration by parts yield that
\[\begin{split}
\int_S \a_S^{-2} \jump{u_h}_h \dv{s} =
\int_S D\phi (\nablah u_h) \cdot n_S \dv{s} 
+ \frac1d \int_S D \psi_h (\Pi_h u_h)  (x-x_T)\cdot n_S \dv{s} 
\end{split}\]
where we used $\int_T \vphi_{S,T} \dv{x} = |T|/(d+1) = (x-x_T)\cdot n_S |S|/d$
for every $x\in S$.
Since $\tz_h \cdot n_\cS$ is constant and continuous on $\cS_h$ we deduce that 
$\a_\cS^{-2} \jump{u_h}_h = \{\tz_h \cdot n_\cS\}$. 
Using the identities 
$\Pi_h \tz_h = D\phi(\nablah u_h)$ and $\diver \tz_h = D\psi_h(\Pi_h u_h)$ 
and noting that these imply equality in~\eqref{eq:fenchel} we find that 
\[\begin{split}
\phi(\nablah u_h) &= \Pi_h \tz_h \cdot \nablah u_h - \phi^*(\Pi_h \tz_h),   \\
\psi_h(\Pi_h u_h) &= \diver \tz_h \, \Pi_h u_h  - \psi_h^*(\diver \tz_h). 
\end{split}\]
Therefore, we have that 
\[\begin{split}
- \int_\O \phi^*(\Pi_h \tz_h)  + \psi_h^*(\diver & \tz_h) \dv{x} 
=  \int_\O \phi(\nablah u_h) + \psi_h(\Pi_h u_h) \dv{x}  \\
& - \int_\O  D\phi(\nablah u_h) \cdot \nablah u_h + D\psi_h(\Pi_h u_h) \Pi_h u_h \dv{x}.
\end{split}\]
Using the discrete Euler--Lagrange equation with 
$v_h = u_h$ we find that
\[
- \int_\O  D\phi(\nablah u_h) \cdot \nablah u_h + D\psi_h(\Pi_h u_h) \Pi_h u_h \dv{x}
= \int_{\cS_h \setminus \GN} \a_\cS^{-2} \jump{u_h}_h^2 \dv{s}.
\]
By combining the last two identities and incorporating
$\a_\cS^{-2} \jump{u_h}_h = \{\tz_h \cdot n_\cS\}$ we deduce that $D_h(\tz_h) = I_h(u_h$).
\end{proof}

\section{Nonlinear Dirichlet problems}\label{sec:nonlin_diri}
We derive an error estimate for a class of nonlinear Dirichlet problems
with linear low order terms. We say that $\phi\in C^1(\R^d)$ is {\em $\s$-coercive}
if there exists a nonnegative functional $\s:\R^d\times \R^d \to \R_{\ge 0}$
such that for all $a,b\in \R^d$ we have
\[
\phi(a) + D\phi(a)\cdot (b-a) + \s(a;b) \le \phi(b).
\]
The low order term is assumed to be given by the function
\[
\psi(x,s) = -f(x) s
\]
for some given $f\in L^{p'}(\O)$. We assume below that the corresponding 
continuous problems, defined with 
\[\begin{split}
I(u) &= \int_\O \phi(\nabla u) \dv{x} - \int_\O f u \dv{x}, \\
D(z) &= -\int_\O \phi^*(z) \dv{x} - I_{-f}(\diver z), 
\end{split}\]
are in strong duality and refer
the reader to~\cite{AtBuMi06-book,Rock70-book} for sufficient conditions and
general statements.
We have that the indicator functional $\psi^*(x,t) = I_{\{-f(x)\}}(t)$ 
enforces the constraint $\diver z = -f$ and that
the discrete primal and dual problem are given by the functionals
\[\begin{split}
I_h(u_h) &= \int_\O \phi(\nablah u_h) \dv{x} -\int_\O f_h u_h \dv{x} + J_h(u_h), \\
D_h(z_h) &= - \int_\O \phi^*(\Pi_h z_h) \dv{x} - I_{\{-f_h\}}(\diver z_h) - K_h(z_h),
\end{split}\]
where we assume $f_h = \Pi_h f$. The indicator functional $I_{\{-f_h\}}$ enforces 
the constraint $\diver z_h = -f_h$. 

\begin{proposition}[Error estimate]\label{prop:error_nonlin_diri}
Assume that $\phi \in C^1(\R^d)$ is strictly convex and $\s$-coercive and assume
that strong duality holds for the continuous problem. 
For the minimizer $u\in W^{1,p}_D(\O)$ of $I$ and the discrete 
minimizer $u_h\in \cS^{1,dg}(\cT_h)$ for $I_h$ we have with a solution
$z\in W^q_N(\diver;\O)$ of the dual problem satisfying the
regularity condition $z\in W^{1,1}(\O;\R^d)$ that
\[\begin{split}
\int_\O \s(\nablah u_h;\nablah \cI_h u) \dv{x}
& \le \int_\O \big(D\phi^*(z)-D\phi^*(\Pi_h\cJ_h z)\big) \cdot (z-\Pi_h\cJ_h z)\dv{x} \\
& \qquad + J_h(\cI_h u) + K_h(\cJ_h z).
\end{split}\]
\end{proposition}

\begin{proof}
The minimality of $u_h$ implies that
\[
\d_h^2 = \int_\O \s(\nablah u_h;\nablah \cI_h u) \dv{x} 
\le I_h(\cI_h u) - I_h(u_h).
\]
With the duality relation $I_h(u_h) \ge D_h(z_h) \ge D_h(\cJ_h z)$ we infer that
\[\begin{split}
\d_h^2 &\le I_h(\cI_h u) - D_h(\cJ_h z) \\
&= \int_\O \phi(\nablah \cI_h u) \dv{x} - \int_\O f_h \Pi_h \cI_h u \dv{x} 
+ \int_\O \phi^*(\Pi_h \cJ_h z) \dv{x} \\
& \qquad + J_h(\cI_h u_h) + K_h (\cJ_h z),
\end{split}\]
where we used that 
$\diver \cJ_h z =-f_h$. Jensen's inequality in combination with 
$\nablah \cI_h u = \Pi_h \nabla u$ and the strong
duality relation $I(u) = D(z)$ lead to 
\[
\int_\O \phi(\nablah \cI_h u) \dv{x} \le \int_\O \phi(\nabla u) \dv{x}
= - \int_\O \phi^*(z) \dv{x} + \int_\O f u \dv{x}.
\]
This implies that we have 
\[\begin{split}
\d_h^2 & \le - \int_\O \phi^*(z) \dv{x} + \int_\O fu - f_h \Pi_h \cI_h u \dv{x} 
+ \int_\O \phi^*(\cJ_h z)\dv{x} \\
&\qquad + J_h(\cI_h u_h) + K_h (\cJ_h z).
\end{split}\]
Since $\diver z= -f$ and $\diver \cJ_h z = -f_h$ it follows from the 
integration-by-parts formula~\eqref{eq:int_parts_dg}
and the identity $\nablah \cI_h u = \Pi_h \nabla u$ that 
\[\begin{split}
\int_\O fu - f_h \Pi_h \cI_h u \dv{x} 
&= \int_\O z \cdot \nabla u  - \cJ_h z  \cdot \nablah \cI_h u \dv{x} \\ 
&=  \int_\O (z- \Pi_h \cJ_h z)  \cdot \nabla  u \dv{x}.
\end{split}\]
We use that $z = D\phi(\nabla u)$ and hence $\nabla u = D\phi^*(z)$, 
i.e., 
\[
\int_\O fu - f_h \Pi_h \cI_h u \dv{x} 
= \int_\O D\phi^*(z) \cdot (z- \Pi_h \cJ_h z) \dv{x}.
\]
The convexity of $\phi^*$ provides the relation 
\begin{equation}\label{eq:conv_phi_star}
\phi^*(\Pi_h \cJ_h z) + D\phi^*(\Pi_h \cJ_h z) \cdot (z-\Pi_h \cJ_h z) 
\le \phi^*(z).
\end{equation}
On combining the inequalities we find that
\[\begin{split}
\d_h^2 & \le 
\int_\O \big(D\phi^*(z)-D\phi^*(\Pi_h\cJ_h z)\big) \cdot (z-\Pi_h \cJ_h z)\dv{x} \\
& \qquad + J_h(\cI_h u_h) + K_h (\cJ_h z),
\end{split}\]
which implies the asserted estimate. 
\end{proof}

\begin{remark}
The estimate of the proposition can be improved by incorporating  
a coercivity property of $\phi^*$ in~\eqref{eq:conv_phi_star}.
\end{remark}

Under additional conditions a convergence rate can be deduced. To illustrate
this we assume for simplicity the Lipschitz property 
\[
\|D\phi^*(v)-D\phi^*(w)\|_{L^q(\O)} \le c_\phi \|v-w\|_{L^q(\O)}
\]
which can be replaced, e.g.,  by a local Lipschitz estimate. 

\begin{corollary}[Lipschitz differentiability]
In addition to the assumptions of Proposition~\ref{prop:error_nonlin_diri} 
assume that $D\phi^*$ is Lipschitz continuous. Then we have  
\[\begin{split}
\int_\O \s(\nablah u_h;& \nablah \cI_h u) \dv{x} 
 \le c_\phi \|z-\Pi_h \cJ_h z\|_{L^q(\O)}^2 \\
& \quad  +  c_\cT \|h_\cS^{-1} \a_\cS^{r'}\|_{L^\infty(\cS)}  \|\cJ_h z \|_{L^{r'}(\O)}^{r'}
+  c_\cT \| h_\cS^{-1} \b_S^s\|_{L^\infty(\cS)}  \|\cI_h u\|_{L^s(\O)}^s.
\end{split}\]
In particular, if $z\in W^{1,q}(\O;\R^d)$, $u\in W^{1,p}(\O)$,
and $\a_\cS = c_\a h_\cS^\g$, $\b_\cS = c_\b h_\cS^\s$ with 
$\g r', \s s \ge 3$ and $r'\le q$, $s\le p$ then the right-hand side
is of quadratic order. 
\end{corollary}

\begin{proof}
The estimate is an immediate consequence of Proposition~\ref{prop:error_nonlin_diri}
noting that $\jump{\cI_h u}_h = 0$ for all $S\in \cS_h\setminus \GN$ and
$\jump{\cJ_h z\cdot n_\cS}_h = 0$ for all $S\in \cS_h \setminus \GD$ and
the inequalities~\eqref{eq:trace_est}. 
\end{proof}

\section{Total-variation minimization}\label{sec:tv_min}
Setting $\GN = \p\O$ and $\GD = \emptyset$ we consider the minimization of the 
functional 
\[
I(u) = |Du|(\O) + \frac{\a}{2}\|u-g\|^2,
\]
in the set of all $u\in BV(\O)\cap L^2(\O)$. We refer the reader 
to~\cite{AmFuPa00-book,Bart15-book} for definitions and 
properties of the variational problem. The dual formulation 
consists in determining $z\in W^2_N(\diver;\O)$ which is maximal
for 
\[
D(z) = -I_{K_1(0)}(z) - \frac{1}{2\a} \|\diver z + \a g\|^2 
+ \frac{\a}{2} \|g\|^2.
\]
In particular, we have that $z\in L^\infty(\O;\R^d)$ and strong
duality applies, i.e., for solutions $u$ and $z$ we have
\[
I(u) = D(z),
\]
cf., e.g.,~\cite{HinKun04}.
The discrete primal problem seeks $u_h\in \cS^{1,dg}(\cT_h)$ 
which is minimial for 
\[
I_h(u_h) = \int_\O |\nablah u_h| \dv{x} 
+ J_h(u_h) + \frac{\a}{2} \| \Pi_h(u_h - g)\|^2. 
\]
The functionals $I_h$ approximate $I$ under moderate conditions
on the discretization parameters $r$ and $p$. 

\begin{proposition}[$\Gamma$-convergence]
Assume $r\ge 1$, $s\le 2$, and that 
\[
\|\a_\cS^{r'} h_\cS^{-1} \|_{L^\infty(\cS_h)}  
+ \|\b_\cS^s h_\cS^{-1}\|_{L^\infty(\cS_h)} \to 0
\]
as $h\to 0$ where the first term is replaced by $\|\a_\cS\|_{L^\infty(\cS_h)}$
if $r=1$. We then have $I_h \to I$ in the sense of
$\G$ convergence with respect to strong convergence in $L^1(\O)$.
\end{proposition}

\begin{proof}
(i) To show that $I(u)\le \liminf I_h(u_h)$ for a sequence
$(u_h)_{h>0}$ with $I_h(u_h) \le c$ we first note that 
\[\begin{split}
|Du_h|(\O) &= \|\nabla u_h\|_{L^1(\O)} + \|\jump{u_h}\|_{L^1(\cS_h\setminus \GN)} \\
&\le \|\nabla u_h\|_{L^1(\O)} + \|\jump{u_h}_h\|_{L^1(\cS_h\setminus \GN)} + \|h_\cS \jump{\nablah u_h}\|_{L^1(\cS_h)} \\
&\le \|\nabla u_h\|_{L^1(\O)} 
+ \|\a_\cS\|_{L^{r'}(\cS_h\setminus \GN)} \|\a_\cS^{-1} \jump{u_h}_h\|_{L^r(\cS_h\setminus \GN)} + c_\cT  \|\nablah u_h\|_{L^1(\O)}.
\end{split}\]
Since $\|\a_\cS\|_{L^{r'}(\cS_h\setminus \GN)}^{r'} 
\le \|\a_\cS^{r'} h_\cS^{-1}\|_{L^{\infty}(\cS_h\setminus \GN)} c_\cT |\O|$ 
and since 
\[
\|u_h\|_{L^1(\O)} = \|\Pi_h u_h\|_{L^1(\O)} + h \|\nablah u_h\|_{L^1(\O)}
\]
we find that $(u_h)_{h>0}$ is bounded in $BV(\O)$. We let $u\in BV(\O)$
be an appropriate accumulation point so that $u_h\to u$ in $L^1(\O)$
and $\Pi_h u_h \wto u$ in $L^2(\O)$. Using that for $\psi\in C_0^\infty(\O;\R^d)$ 
we have 
\[\begin{split}
\int_\O \Pi_h u_h \diver \psi \dv{x} =  & - \int_\O \nablah u_h \cdot \psi \dv{x}   
+ \int_{\cS_h \setminus \GN} \jump{u_h}_h \cJ_h \psi \cdot n_\cS \dv{s}  \\
& + \int_\O \nablah u_h \cdot (\psi- \cJ_h \psi) \dv{x},
\end{split}\]
where
\[
 \int_{\cS_h} \jump{u_h}_h \cJ_h \psi \dv{s} \le \|\a_\cS\|_{L^{r'}(\cS_h)}
\|\a_\cS^{-1} \jump{u_h}_h\|_{L^r(\cS_h\setminus \GN)} \|\cJ_h \psi \|_{L^\infty(\cS_h)}
\]
tends to zero owing to the conditions on $\a_\cS$. If $\|\psi\|_{L^\infty(\O)} \le 1$ then
this leads to 
\[
\int_\O u \diver \psi \dv{x} \le \liminf_{h\to 0} \|\nablah u_h\|_{L^1(\O)}
\]
and in particular to the bound
\[
|\DD u|(\O) \le \liminf_{h\to 0} \|\nablah u_h\|_{L^1(\O)}.
\]
Since $\Pi_h(u_h-g_h) \wto (u-g)$ in $L^2(\O)$ and $J_h(u_h)\ge 0$ we deduce that 
$I(u)\le \liminf_{h\to 0} I_h(u_h)$.  \\
(ii) To prove that for every $u \in BV(\O)\cap L^2(\O)$
there exists a sequence $(u_h)_{h>0}$ with $u_h \wto u$ in $L^2(\O)$
and $I(u) = \lim_{h \to 0} I_h(u_h)$ 
we use the intermediate density of continuous finite element
functions in $BV(\O)\cap L^2(\O)$ to obtain a sequence $(u_h)_{h>0}$ with
$\|\a_\cS \jump{u_h}\|_{L^r(\cS_h\setminus \GN)} =0$ and which converges
intermediately in $BV(\O)$, weakly in $L^2(\O)$, and strongly in $L^1(\O)$
to~$u$. The condition on $\b_\cS$ yields that 
$\|\b_\cS \{u_h\}\|_{L^s(\cS_h)} \to 0$ as $h\to 0$. 
Altogether, this implies the attainment result $I(u) = \lim_{h\to 0} I_h(u_h)$.
\end{proof}

\begin{remark}
For $r=1$ the condition $\|\a_\cS\|_{L^\infty(\cS_h)} \to 0$ corresponds to
the use of quadrature in the definition of the penalty terms. 
If instead of the mean of the jump $\jump{u_h}_h$ the full jump
$\jump{u_h}$ is used in the definition of
$J_h$, then if $r=1$ it suffices to require that $\a_\cS =1$ since the functional
$I_h$ then involves the exact term $|Du_h|(\O)$. Our error estimate below
shows that the condition on $\a_\cS$ can be weakened if a regularity
condition is satisfied. 
\end{remark}

For an error estimate the discrete dual functional is required. 
It consists in maximizing the functional
\[
D_h(z_h) = - I_{K_1(0)} (\Pi_{0,h} z_h)  
- \frac{1}{2\a} \|\diver z_h + \a g_h\|^2 + \frac{\a}{2} \|g_h\|^2 - K_h(z_h)
\]
in the set of vector fields $z_h\in \RT^{0,dg}(\cT_h)$. 

\begin{proposition}[Error estimate]\label{prop:tv_err_est}
Assume that $g\in L^\infty(\O)$ and that there exists a Lipschitz 
continuous solution $z\in W^2_N(\diver;\O) \cap W^{1,\infty}(\O)$ for 
the dual problem. Moreover, suppose that 
\[
\|h_\cS^{-1} \a_\cS^{r'} \|_{L^\infty(\cS_h)}
+ \|h_\cS^{-1} \b_\cS^s\|_{L^\infty(\cS_h)} \le c h,
\]
where the first term can be omitted if $r=1$ and $0<\a_\cS \le 1$. Then, for the 
solutions $u\in BV(\O)\cap L^2(\O)$ and $u_h\in \cS^{1,dg}(\cT_h)$
of the primal and discrete primal problem we have 
\[
\| u- \Pi_h u_h \|\le c h^{1/2} M_{u,z,g},
\]
with a factor $M_{u,z,g}$ that depends on $\a>0$, 
$\|u\|_{L^\infty(\O)} \le \|g\|_{L^\infty(\O)}$, 
$\|g\|_{L^2(\O)}$, and $\|\nabla z\|_{L^\infty(\O)}$. 
\end{proposition}

\begin{proof}
(i) By the coercivity of the discrete functional $I_h$ and the
discrete duality relation $\inf I_h \ge \sup D_h$ we have 
\[
\frac{\a}{2} \|\Pi_h(v_h-u_h)\|^2 \le I_h(v_h)- I_h(u_h) 
\le I_h(v_h) -D_h(y_h) 
\]
for every $v_h\in \cS^{1,dg}(\cT_h)$ and $y_h \in \RT^{0,dg}(\cT_h)$. \\
(ii) By noting that $u\in L^\infty(\O)$ and choosing regularizations
$(u_\veps)_\veps >0$ of $u$ 
we construct a quasi-interpolant 
$\tu_h = \lim_{\veps \to 0} \cI_h u_\veps \in \cS^{1,cr}(\cT_h)$ satisfying  
\[\begin{split}
\|\nabla_h \tu_h\|_{L^1(\O)} &\le |\DD u|(\O),  \\
\|\tu_h \|_{L^\infty(\O)} &\le  c_d \|u\|_{L^\infty(\O)},  \\
\|\tu_h  - u\|_{L^1(\O)} &\le c_{\cI,1} h |\DD u|(\O). 
\end{split}\]
In particular, we have that $\jump{\tu_h}_h = 0$ on inner element sides
$S\in \cS_h \setminus \p\O$ and hence
\[\begin{split}
I_h(\tu_h) &  = \|\nablah \tu_h\|_{L^1(\O)} 
 + \frac{\a}{2} \| \Pi_h(\tu_h - g)\|^2 + \frac{1}{s} \|\b_\cS \{\tu_h\}_h \|_{L^s(\cS_h)}^s \\
&\le I(u) + \frac{\a}{2} \big(\|\Pi_h(\tu_h - g)\|^2 - \|u - g\|^2\big)
+ \frac{1}{s} \|h_\cS^{-1} \b_\cS^s\|_{L^\infty(\cS_h)} c_\cT \|\tu_h\|_{L^s(\O)}^s.
\end{split}\]
Abbreviating $\ou_h = \Pi_h \tu_h$ and $g_h = \Pi_hg$ we have 
\[\begin{split}
\|\ou_h-g_h\|^2 &= \|\ou_h-g\|^2 - \|g-g_h\|^2 \\
& = \|u- g\|^2  + \int_\O (\ou_h-u) (\ou_h + u -2g)\dv{x} - \|g-g_h\|^2 .
\end{split}\]
Incorporating the bound $\|\cI_h u\|_{L^s(\O)} \le c_s \|u\|_{L^\infty(\O)}$,
these identities imply that 
\[\begin{split}
I_h(\tu_h)  &\le I(u) 
+ \frac{\a}{2} \|\ou_h-u\|_{L^1(\O)} \|\ou_h+u -2g \|_{L^\infty(\O)}  \\
&\qquad -  \frac{\a}{2} \|g-g_h\|^2 
+ c_\cT c_s^s \|h_\cS^{-1} \b_\cS^s\|_{L^\infty(\cS_h)} \|u\|_{L^\infty(\O)}^s. 
\end{split}\]
(iii) With $L = \|\nabla z\|_{L^\infty(\O)}$ we have 
that $\|\cJ_h z\|_{L^\infty(\O)} \le \vrho_h= 1 + c h L$. Hence, 
for $\tz_h = \vrho_h^{-1} \cJ_h z \in \RT^0_{\!N}(\cT_h)$ we have 
$|\Pi_h\tz_h(x_T)|\le 1$
as well as $\diver \tz_h = \vrho_h^{-1} \Pi_h \diver z$. With these relations
we deduce that
\[\begin{split}
-D_h(\tz_h) &=  K_h(\tz_h) + \frac{1}{2\a} \|\diver \tz_h + \a g_h\|^2 - \frac{\a}{2} \|g_h\|^2 \\
& \le  K_h(\tz_h) + \frac{1}{2\a} \| \vrho_h^{-1} \diver z + \a g\|^2 - \frac{\a}{2} \|g_h\|^2 \\
& =  K_h(\tz_h) + \frac{\vrho_h^{-2}}{2\a} \| \diver z \|^2 + \vrho_h^{-1} \int_\O \diver z  g \dv{x} 
 + \frac{\a}{2} \big(\|g\|^2-\|g_h\|^2\big) \\
& \le - D(z) + K_h(\tz_h) + \frac{\a}{2} \|g-g_h\|^2 + |\vrho_h^{-1}-1| \|\diver z\| \|g\|,
\end{split}\]
where we used Jensen's inequality, $\vrho_h\ge 1$, and $\|g_h\|^2-\|g\|^2 = \|g-g_h\|^2$.
In the case $r=1$ we note that $|z|\le 1$ implies that $|\tz_h \cdot n_\cS| \le 1$
and hence since $\a_\cS^{-1} \ge 1$ that $K_h(\tz_h)  = 0$. 
If $r>1$ we have
\[
K_h(\tz_h) =  \frac{1}{r'} \| \a_S \{\tz_h \cdot n_\cS\} \|_{L^{r'}(\cS_h)}^{r'} 
\le c_\cT \|h_\cS^{-1} \a_\cS^{r'}\|_{L^\infty(\cS_h)} \|z \|_{L^{r'}(\O)}^{r'}.
\]
(iv) We are now in position to combine the previous estimates. The
choices $v_h = \tu_h$ and $y_h = \tz_h$ lead to 
\[\begin{split}
\frac{\a}{2} \|\Pi_h(\tu_h-&u_h)  \|^2 \le I_h(\tu_h) - D_h(\tz_h) \\
&\le I(u) + \frac{\a}{2} \|\ou_h-u\|_{L^1(\O)} \|\ou_h+u -2g \|_{L^\infty(\O)}  \\
& \quad - D(z) + |\vrho_h^{-1}-1| \|\diver z\| \|g\| \\
& \quad +  c_\cT \d_r \|h_\cS^{-1} \a_\cS^{r'}\|_{L^\infty(\cS_h)} \|z \|_{L^{r'}(\O)}^{r'}
+  c_\cT \|h_\cS^{-1} \b_\cS^s\|_{L^\infty(\cS_h)} \|u\|_{L^s(\O)}^s,
\end{split}\]
where $\d_r = 0$ if $r=1$ and $\d_r =1$ otherwise. 
Using $I(u)=D(z)$, the estimate $1-\vrho_h^{-1} \le c h L$,
the approximation properties of $\tu_h$, and the conditions of the
proposition show that 
\[
\frac{\a}{2} \|\Pi_h(\tu_h-u_h)\|^2 \le  c h \widetilde{M}_{u,z,g}^2.
\]
(v) With the estimate
\[\begin{split}
\|u-\Pi_h \tu_h\|^2 &\le \|u-\Pi_h \tu_h\|_{L^\infty(\O)} \|u-\Pi_h \tu_h\|_{L^1(\O)}  \\
&\le (1+c_d)\|u\|_{L^\infty(\O)}  \big(\|u-\tu_h\|_{L^1(\O)} + \|\tu_h - \Pi_h \tu_h\|_{L^1(\O)}\big) \\
&\le (1+c_d) \|u\|_{L^\infty(\O)}  h \big( c |Du|(\O) + \|\nablah \tu_h \|_{L^1(\O)}\big),
\end{split}\]
we deduce the asserted error bound. 
\end{proof}

\begin{remarks}
(i) If $(u_h)_{h>0}$ is uniformly bounded in $L^\infty(\O)$
then using that 
\[
\|u_h -\Pi_h u_h\| \le 2 h^{1/2} \|\nablah u_h\|_{L^1(\O)} \|u_h\|_{L^\infty(\O)}
\]
we may replace $\Pi_h u_h$ by $u_h$ in the error estimate of 
Proposition~\ref{prop:tv_err_est}. \\
(ii) A reduced convergence rate is expected if $z$ fails to be Lipschitz
continuous. If only $u\in L^\infty(\O)\cap BV(\O$ is assumed then a
convergence rate $\cO(h^{1/4})$ can be established, 
cf.~\cite{BaNoSa14,ChaPoc19-pre,Bart20-pre}.
\end{remarks}

\section{Obstacle problem}\label{sec:obstacle}
A model obstacle problem is defined by the functional
\[
I(u) = \frac12 \int_\O |\nabla u|^2 \dv{x} - \int_\O f u \dv{x} + I_{\R_{\ge 0}}(u)
\]
for $u\in W^{1,p}_D(\O)$. The dual functional is given by
\[
D(z) = - \frac12 \int_\O |z|^2 \dv{x} - I_{\R_{\le 0}}(f+\diver z)
\]
for vector fields $z\in W^2_N(\diver;\O)$. We have the strong duality 
relation $I(u)=D(z)$ for solutions $u$ and $z$ and the pointwise complementarity
principle that if $u>0$ then $f+\diver z=0$. The discrete functionals are given by
\[
I_h(u_h) = \frac12 \int_\O |\nablah u_h|^2 \dv{x} - \int_\O f_h \Pi_h u_h \dv{x}
+ I_{\R_{\ge 0}}(\Pi_h u_h) + J_h(u_h),
\]
and 
\[
D_h(z_h) =
-\frac12 \int_\O |\Pi_h z_h|^2 \dv{x} - I_{\R_{\le 0}} (f_h + \diver z_h) - K_h(z_h).
\]
Owing to Theorem~\ref{thm:discr_duality} we have that $I_h(u_h) \ge D_h(z_h)$.
We assume that the functionals $J_h$ and $K_h$ are defined
with the parameters and quantities
\[
r = s = 2, \quad \a_\cS = c_\a h_\cS^\g, \quad \b_\cS = c_\b h_\cS^\g
\]
for parameters $\g\ge 3/2$, $\a>0$, and $\b\ge 0$. 

\begin{proposition}[Error estimate]
Assume that $u\in W^{1,2}_D(\O) \cap W^{2,2}(\O)$. Then we have that 
\[\begin{split}
\|\nablah (u_h- u)\|^2 & \le c_\cJ^2 h^2 \|D z\|^2 
+ 2 c_\cI^2 h^2 \|f+\diver z\| \|D^2 u\| \\
& \quad + c_\cT c_{\a,\b}^2 h^{2\g -1} 
\big(\|u\|_{W^{1,2}(\O)}^2 + \|z\|_{W^{1,2}(\O)}^2\big).
\end{split}\]
\end{proposition}

\begin{proof}
We first note that the quasi-interpolants $\cI_hu$ and $\cJ_hz$
are well defined and admissible in the discrete primal and dual
problems, respectively, i.e., we have 
\[
\cI_hu(x_T) = \frac{1}{d+1} \sum_{S\subset \p\T} \int_S u(s) \dv{s}  \ge 0,
\]
for every $T\in \cT_h$ and $f_h+\diver\cJ_h z = \Pi_h(f+\diver z) \le 0$. 
The coercivity of $I_h$ and the discrete duality relation 
$I_h(u_h) \ge D_h(\cJ_hz)$ lead to 
\[
\d_h^2 = \frac12 \|\nablah (u_h-\cI_hu)\|^2
\le I_h(\cI_hu) - I_h(u_h) \le I_h(\cI_hu) - D_h(\cJ_hz).
\]
By Jensen's inequality and $\nablah \cI_hu = \Pi_h \nabla u$ 
we have $\|\nablah \cI_hu\|\le \|\nabla u\|$ and with
the strong duality relation $I(u)=D(z)$ we infer that
\[\begin{split}
\d_h^2 &\le \frac12 \|\nabla u\|^2 - (f_h,\cI_h u) + J_h(\cI_hu)
+ \frac12 \|\Pi_h \cJ_hz\|^2 + K_h(\cJ_hz) \\
&= - \frac12 \|z\|^2 + (f,u) -(f_h,\Pi_h\cI_h u_h) + J_h(\cI_hu)
+ \frac12 \|\Pi_h \cJ_hz\|^2 + K_h(\cJ_hz).
\end{split}\]
The binomial formula $a^2-b^2 = 2 b(a-b) + (a-b)^2$ and
the identities $f_h = \Pi_h f$ and $z=\nabla u$ lead to the estimate 
\[\begin{split}
\d_h^2 & \le (z,\Pi_h \cJ_h z -z) + \frac12 \|\Pi_h \cJ_h z -z\|^2
+ (f,u - \Pi_h\cI_h u_h) + J_h(\cI_hu) + K_h(\cJ_hz) \\
&= (\nabla u,\Pi_h \cJ_h z -z) + \frac12 \|\Pi_h \cJ_h z -z\|^2
+ (f,u - \Pi_h\cI_h u) + J_h(\cI_hu) + K_h(\cJ_hz).
\end{split}\]
With the relation $\Pi_h \nabla u = \nablah \cI_h u$, an integration
by parts, and $\diver \cJ_h z = \Pi_h \diver z$ we obtain the identities
\[
(\nabla u,\Pi_h \cJ_h z -z)
= (\nablah \cI_h u,\cJ_h z) - (\nabla u,z) 
= (\diver z,u-\Pi_h \cI_h u).
\]
Using this and the abbreviation $\l = f+\diver z$ show that we have
\[\begin{split}
\d_h^2 & \le \frac12 \|\Pi_h \cJ_h z -z\|^2
+ (f + \diver z, u-\Pi_h \cI_h u) + J_h(\cI_hu) + K_h(\cJ_hz) \\
&= \frac12 \|\Pi_h \cJ_h z -z\|^2
+ (\l, u-\cI_h u) + (\l,\cI_h u - \Pi_h \cI_h u) + J_h(\cI_hu) + K_h(\cJ_hz)
\end{split}\]
We note that $\cI_h u|_T - \Pi_h \cI_h u(x_T) = \nablah \cI_h u|_T \cdot (x-x_T)$
and that on the element contact set 
\[
\cC_T = \{x\in T: u(x) = 0\}
\]
we have $\nabla u|_{\cC_T} = 0$ and $\l|_{T\setminus \cC_T} = 0$. Hence, 
it follows that
\[
\int_T \l (\cI_h u - \Pi_h \cI_h u)\dv{x}
= \int_{\cC_T} \l \, (x-x_T)\cdot \nabla (\cI_h u - u) \dv{x} 
\]
for every $T\in \cT_h$. We thus obtain the estimate
\[\begin{split}
\d_h^2 & \le \frac12 \|\Pi_h \cJ_h z -z\|^2
+ \|\l\| \big(\|u-\cI_h u\| + h \|\nablah (u-\cI_h u)\|\big) \\
& \qquad +  J_h(\cI_hu) + K_h(\cJ_hz).
\end{split}\]
For the side functionals $J_h$ and $K_h$ we have, owing to the
continuity properties of $\cI_h u$ and $\cJ_h z$ that
\[\begin{split}
J_h(\cI_h u) + K_h(\cJ_hz) 
&= \frac{c_\a^2}{2} \|h_\cS^\g \{\cI_h u\}_h\|_{L^2(\cS_h\setminus \GD)}^2
+ \frac{c_\b^2}{2} \|h_\cS^\g \{\cJ_h z\cdot n_\cS\}\|_{L^2(\cS_h \setminus \GN)}^2 \\
&\le \frac12 c_\cT c_{\a,\b}^2 h^{2\g -1} \big(\|\cI_h u \|^2 + \|\cJ_h z\|^2 \big).
\end{split}\]
By combining the previous estimates we arrive at
\[\begin{split}
\d_h^2 & \le \frac12 \|\Pi_h \cJ_h z -z\|^2
+ \|\l\| \big(\|u-\cI_h u\| + h \|\nablah (u-\cI_h u)\|\big) \\
& \qquad + \frac12 c_\cT h^{2\g -1} c_{\a,\b}^2\big( \|\cI_h u \|^2 + \|\cJ_h z\|^2 \big).
\end{split}\]
With basic stability properties of the quasi-interpolation operators
as operators from $W^{1,2}(\O;\R^\ell)\to L^2(\O;\R^\ell)$
we deduce the asserted error bound.
\end{proof}

\begin{remark}
By defining discontinuous Galerkin methods with certain consistency
properties it is possible to derive optimal convergence rates 
with a penalty term that only involves the factor $h^{-1}$, cf.~\cite{WaHaCh10}.
The approach followed here applies to a large class of variational
problems and allows for a simple error analysis. 
\end{remark}

\section{Numerical experiments}\label{sec:num_ex}
We verify in this section the theoretical results and discuss the 
role of the parameters involved in the discontinuous Galerkin
discretizations. 

\subsection{Poisson problem}
To verify the optimality of the conditions on the weight 
function $\a_\cS$ in the error estimates we consider a Poisson 
problem. The discretized functional reads
\[
I_h(u_h) = \frac12 \int_\O |\nablah u_h|^2 \dv{x}
- \int_\O f_h u_h \dv{x} 
+ \frac{c_\a^{-2}}{2} \int_\cS h_\cS^{-2\g} |\jump{u_h}|^2 \dv{s},
\]
subject to homogeneous Dirichlet boundary conditions 
for $u_h$ on $\GD = \p\O$. Our parameters correspond to the
settings
\[
\a_\cS = c_\a h_\cS^\g, \quad \b_\cS = 0, \quad r = 2, \quad s = 2,
\]
where we consider combinations of the parameters 
\[
\g \in \{0.5,1.0,1.5,2.0\}, \quad c_\a^{-1} \in \{1.0, 4.0\}.
\]

\begin{example}\label{ex:poisson}
Let $d=2$, $\O=(-1,1)^2$, $\GD = \p\O$,
and for $x=(x_1,x_2)\in \O$ set
\[
f(x_1,x_2) = 2 \pi^2 \sin(\pi x_1) \sin(\pi x_2).
\]
Then, the exact solution is given by 
\[
u(x_1,x_2) = \sin(\pi x_1) \sin(\pi x_2)
\]
and satisfies $u\in W^{1,2}_D(\O)\cap W^{2,2}(\O)$. 
\end{example}

The plots in Figure~\ref{fig:poisson_eoc} show the 
experimental errors
\[
\|\nablah e_h\| = \|\nablah (u-u_h)\|
\]
versus the number of elements $N= \# \cT_h \sim h^{-2}$
for different combinations of parameters $\g$ and
$c_\a$. We obsere that the choices $\g=1/2$ and $\g=1$ 
do in not lead to an experimental optimal convergence
rate. The choice $\g=3/2$ leads to linear convergence
independently of the choice of the constant factor $c_\a$
which is in agreement with the theoretical error estimates. 

\begin{figure}[h!]
\includegraphics[width=9.2cm]{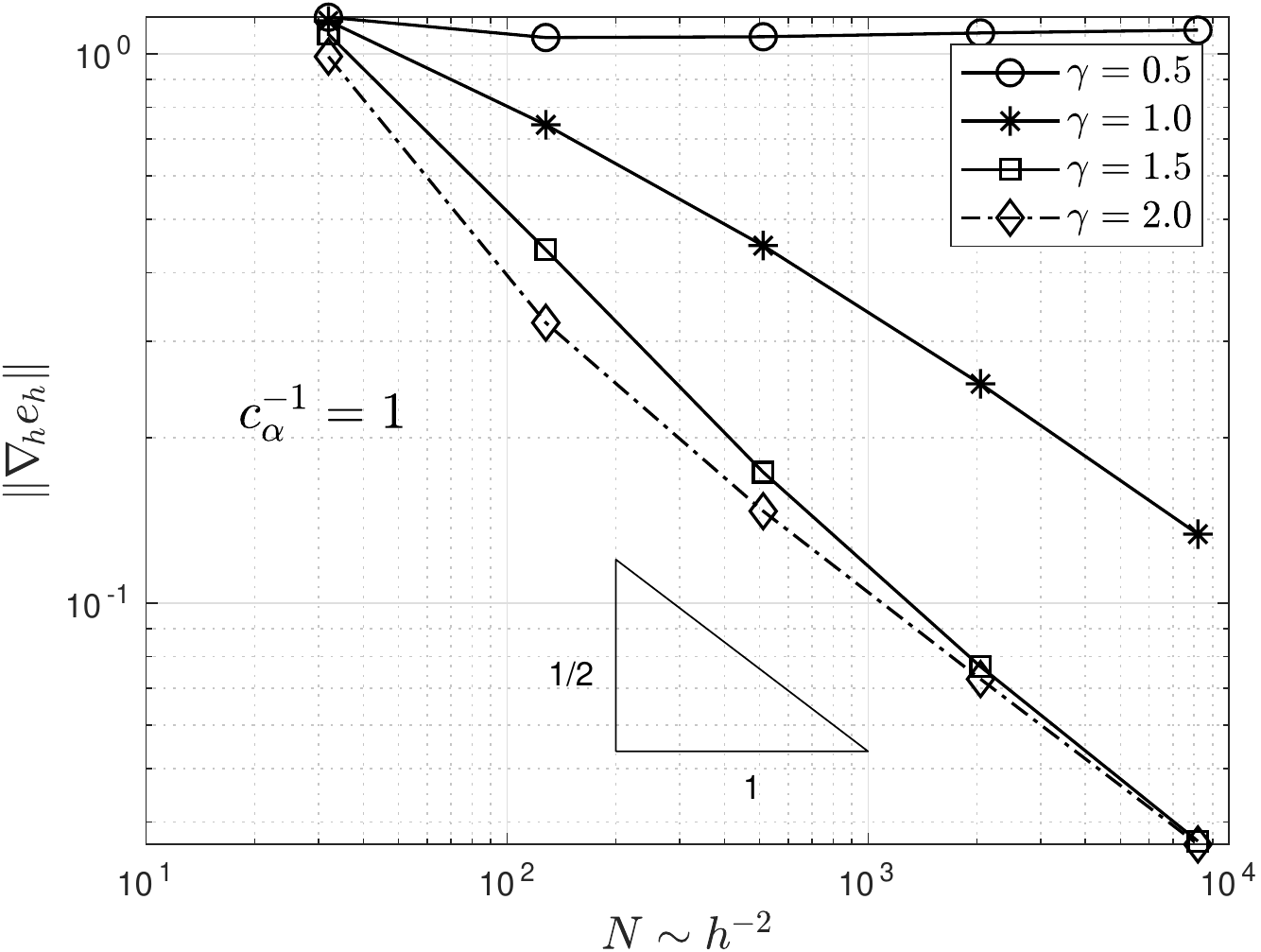} \\[5mm]
\includegraphics[width=9.2cm]{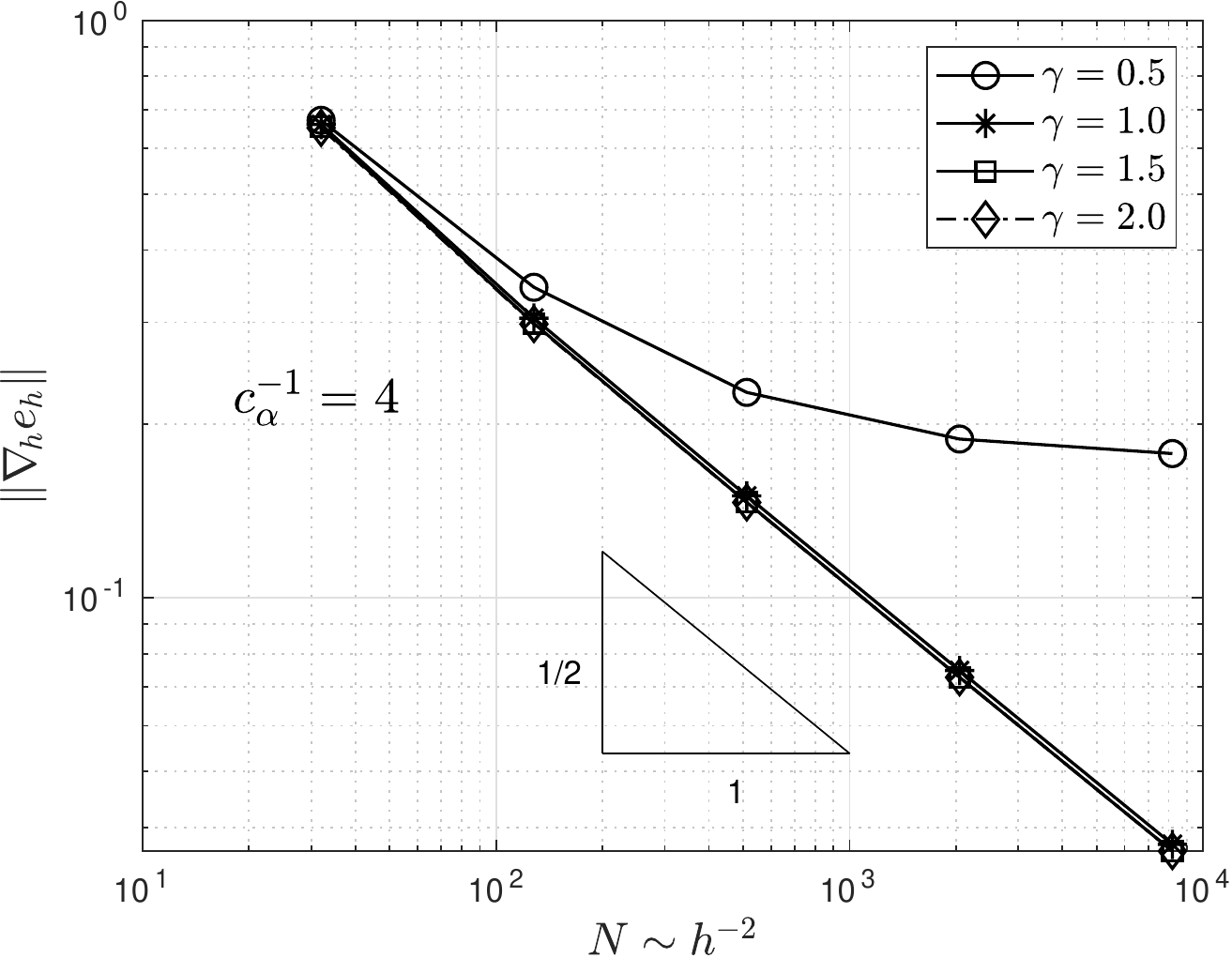}
\caption{\label{fig:poisson_eoc} Experimental convergence rates
in the approximation of the Poisson problem defined in 
Example~\ref{ex:poisson} for different penalty functionals.}
\end{figure}

\subsection{Total-variation minimization}
For a given triangulation we consider the discrete minimization
problem defined via the functional
\[
I_{h,\veps}(u_h) = \int_\O |\nablah u_h|_\veps \dv{x}
+ \frac{\a}{2} \|\Pi_h u_h - g_h\|^2
+ \frac{c_\a^{-r}}{r} \int_\cS h_\cS^{-\g r} |\jump{u_h}|_\veps^r \dv{s}
\]
with the regularized modulus or length $|a|_\veps = (|a|^2 +\veps^2)^{1/2}$
for $a\in \R^\ell$ and $\veps>0$.
Since $0\le |a|_\veps- |a| \le \veps$ the error estimate
of Proposition~\ref{prop:tv_err_est} remains valid provided that
$\veps \le c h$, we therefore choose $\veps = h$. 
The definition correponds to the settings
\[
\a_\cS = c_\a h_\cS^\g, \quad \b_\cS = 0.
\]
In the following
example we consider Dirichlet boundary conditions on $\GD = \p\O$.
While a general existence theory is lacking our error analysis
remains valid provided a solution exists, which is the 
case for the considered setting.

\begin{example}\label{ex:tv}
For $\O\subset \R^d$, $\a>0$, and $R>0$ such that 
$\overline{B_R(0)}\subset \O$, let
\[
g(x) = \chi_{B_R(0)}(x).
\]
Then $u(x) = \max\{1-2/(\a R),0\} \chi_{B_R(x)}$ is the unique solution
of the total variation minimization problem subject to
homogeneous Dirichlet conditions on $\GD =\p\O$.
The solution $z\in W^2(\diver;\O)$ of the dual problem is 
given by
\[
z(x) = 
\begin{cases}
R^{-1} x & \mbox{for }|x|\le R, \\
-R x/|x|^2 & \mbox{for } |x|\ge R,
\end{cases}
\]
and satisfies $z\in W^{1,\infty}(\O;\R^d)$. We set $\a=10$,
$R=1/2$, and $\O = (-1,1)^2$. 
\end{example}

Our numerical approximations are obtained with a
semi-implicit discretization of an $L^2$ gradient flow 
for $I_{h,\veps}$ with step-size $\tau=1$ and $L^2$ 
stopping criterion $\veps_{\rm stop} = h/100$. We refer
the reader to~\cite{ChaPoc11,Bart15-book} for discussions
of iterative methods. 
The top and bottom plots in Figure~\ref{fig:sols_tv} show 
numerical solutions for the parameters
\[
\text{(t)}\quad r = 1,\ \g = 1, \ c_\a = 1, \qquad
\text{(b)}\quad r = 2,\ \g = 1, \ c_\a = 1,
\] 
on the triangulations $\cT_\ell$ with $\ell=4$ 
consisting of $2^\ell$ halved squares.  We observe that 
the choice $r=2$ leads to an artificially rounded region,
according to the error analysis of Proposition~\ref{prop:tv_err_est}
they are of comparable accuracy. The analysis showed
that the error bound is independent of the $\g$ and
$c_\a$ if $r=1$. This is confirmed by the 
experimental convergence rates shown in Figure~\ref{fig:eoc_tv},
where the error quantity
\[
\|e_h\|^2 = \|\Pi_h (u-u_h)\|^2
\]
is plotted against the number of elements in $\cT_\ell$ 
for combinations of the parameters $r\in \{1,2\}$ and 
$\g \in \{0, 1, 2\}$ and $c_\a^{-1} = 10$. 
We observe the expected rate $h^{1/2}$ for all combinations
except when $\g=0$. In the case $\g=0$ we only observe an error
decay if $r=1$ which confirms the theoretical results but does not 
lead to the expected optimal convergence 
rate. Further experiments indicated that this is
related to the use of regularization
and the approximate iterative solution of the nonlinear systems.

\begin{figure}[h!]
\includegraphics[width=11cm]{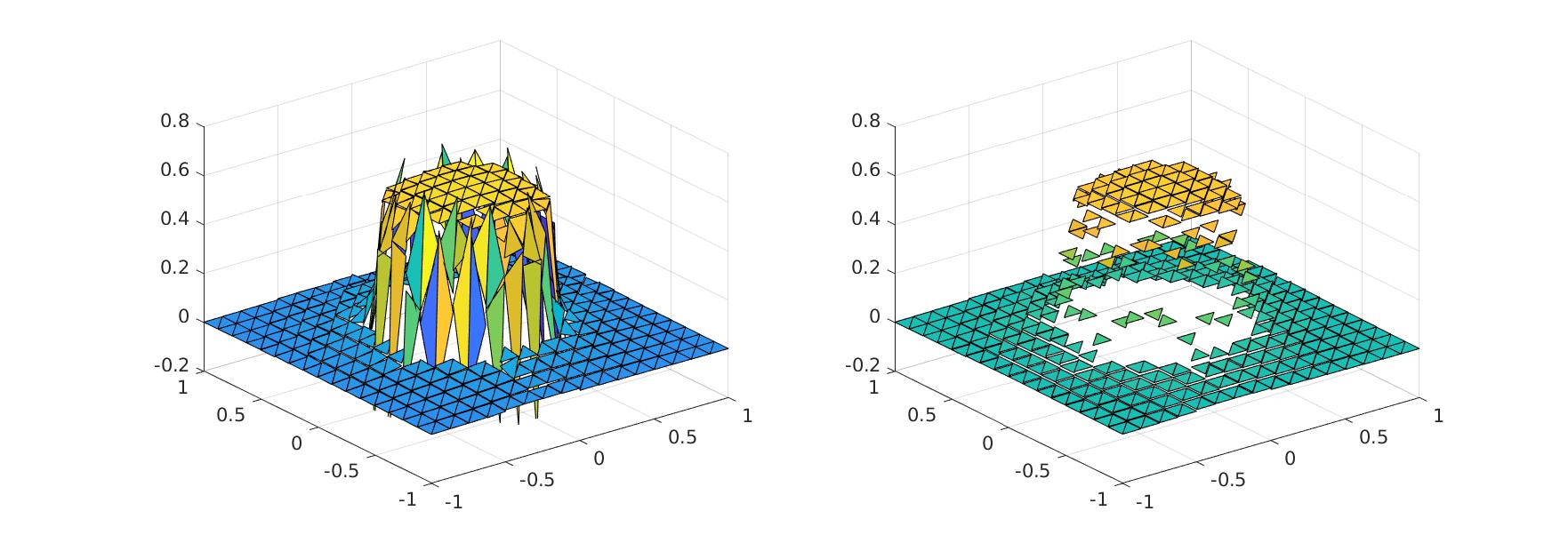}
\includegraphics[width=11cm]{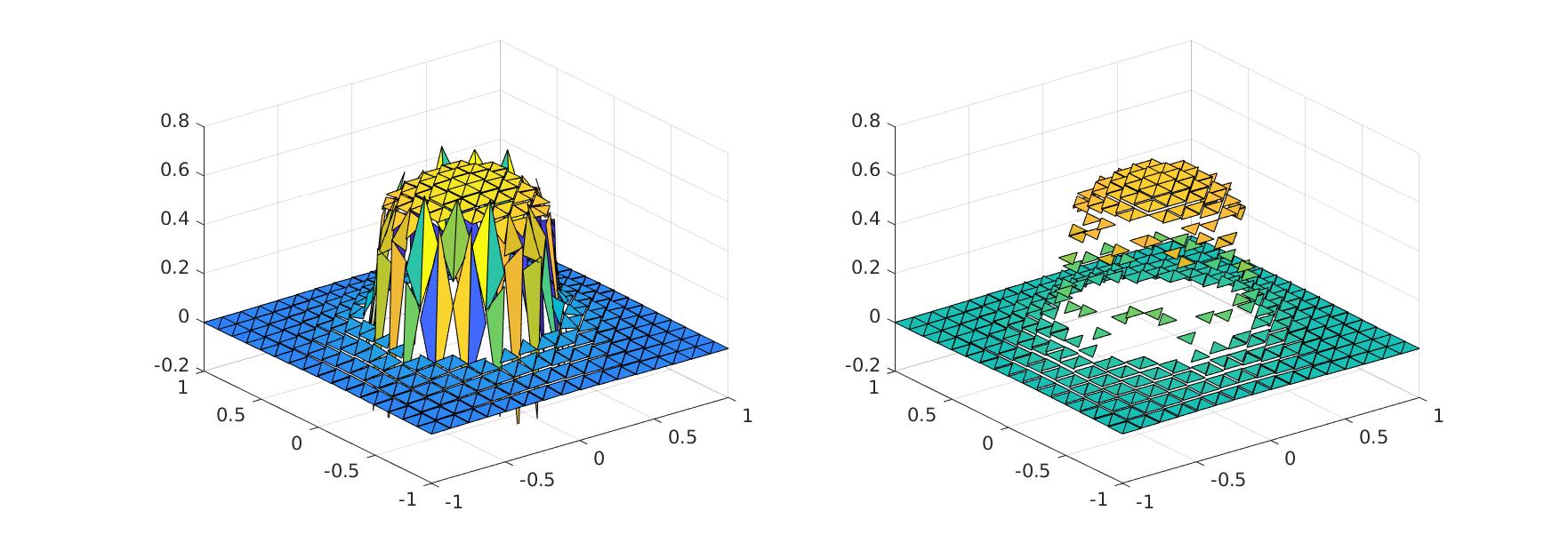}
\caption{\label{fig:sols_tv} Approximations $u_h$ and projections $\Pi_h u_h$
for the total variation minimization problem defined in 
Example~\ref{ex:tv} for linear (top) and quadratic (bottom) 
penalty terms.}
\end{figure}

\begin{figure}[h!]
\includegraphics[width=9.2cm]{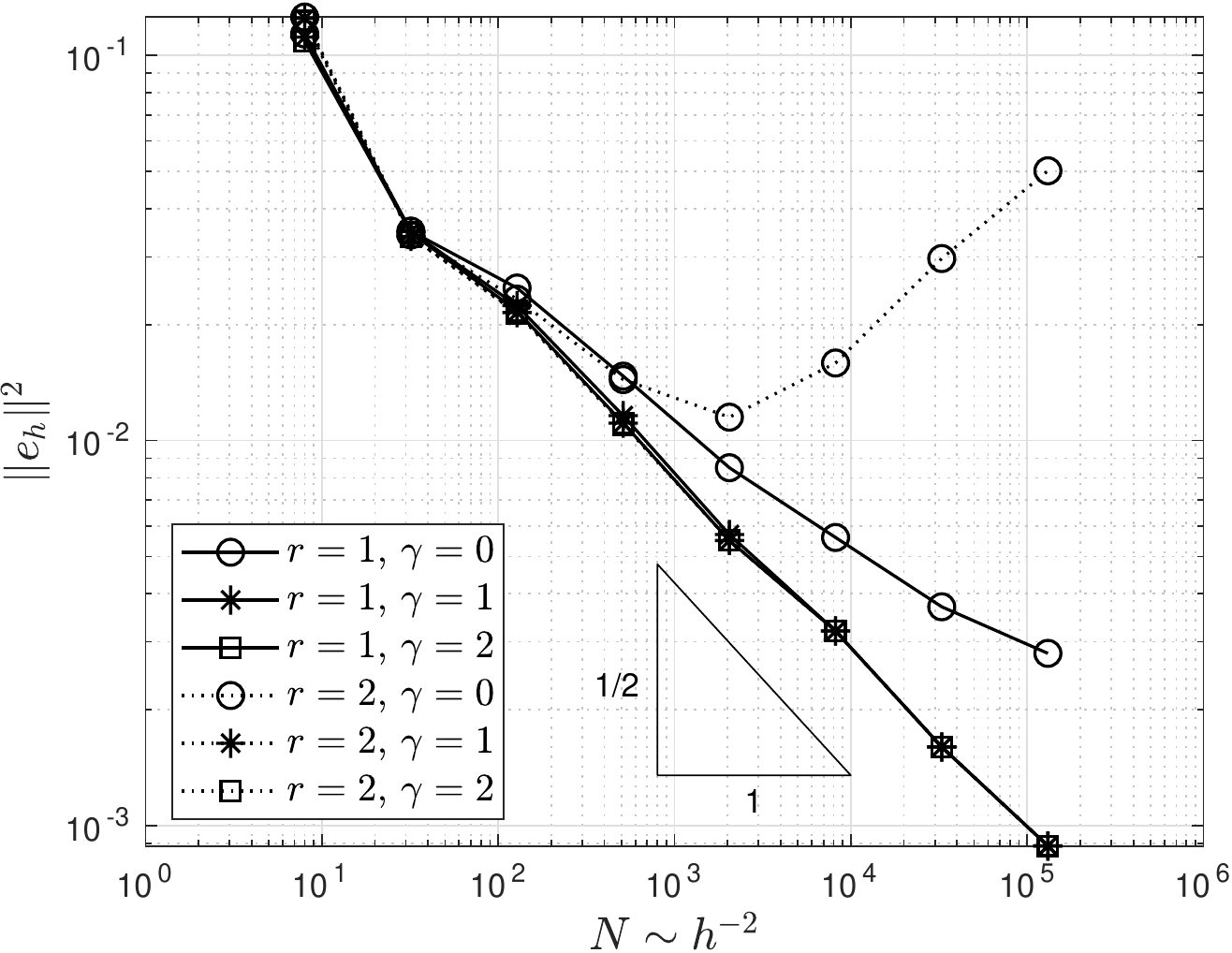}
\caption{\label{fig:eoc_tv} Experimental convergence rates
in the approximation of the total variation minimization problem
defined in Example~\ref{ex:tv} for different penalty terms.}
\end{figure}

\subsection{Obstacle problem}
We consider an obstacle problem that includes inhomogeneous
Dirichlet boundary conditions via a decomposition of the
solution and thus leads to the discrete functional
\[\begin{split}
I_h(u_h) &= \frac12 \int_\O |\nablah u_h|^2 \dv{x}
- \int_\O f_h u_h \dv{x} + I_{\tchi_h}(\Pi_h u_h) \\
&\qquad  + \int_\O \nablah \cI_h \tu_D \cdot \nablah u_h \dv{x}
+ \frac{c_\a^{-2}}{2} \int_\cS h_\cS^{-2\g} |\jump{u_h}|^2 \dv{s},
\end{split}\]
with the transformed obstacle $\tchi_h = \Pi_h(\chi - \cI_h \tu_D)$
and subject to homogeneous Dirichlet boundary conditions 
for $u_h$ on $\GD = \p\O$. The approximate solution is
thus $u_h + u_{D,h}$. Our parameters correspond to the
settings
\[
\a_\cS = c_\a h^{-\g}, \quad \b_\cS = 0, \quad r = 2, \quad s = 2.
\]
We specify the data in following example. 

\begin{example}[\cite{LiLiTa00}]\label{ex:obst}
Let $\O=(-3/2,3/2)^2$, $f=-2$, $\chi = 0$, and 
$u_D(x) = |x|^2/2 - \log(|x|) - 1/2$ for $x\in \GD = \p\O$. 
Then, the exact solution is given by 
\[
u(x) =
\begin{cases} 
|x|^2/2 - \log(|x|) - 1/2 & \mbox{for } |x|\ge 1, \\
0 & \mbox{for } |x|\le 1.
\end{cases}
\]
and satisfies $u\in W^{1,2}_D(\O)\cap W^{2,2}(\O)$. 
\end{example}

We solved the discrete minimization problem with a semismooth
Newton iteration as in~\cite{HinKun02} that converged 
superlinearly towards the
stopping criterion that required a correction in the discrete
$H^1$ norm less than $\veps_{\rm stop} = h$. 
The left and right plots of Figure~\ref{fig:obst_red_4} show 
the discontinuous Galerkin
approximations for the penalty functionals defined via
\[
\text{($\ell$)}\quad r = 2, \g = 1, \ c_\a = 1, \qquad
\text{(r)}\quad r = 2, \g = 3/2, \ c_\a = 1.
\] 
We observe that the jumps along inner edges are smaller
for the larger exponent $\g$. The factor $c_\a$ strongly
influences the preasymptotic range of the convergence
rate which can be observed from Figure~\ref{fig:obst_eoc}
where we plotted the approximation errors 
\[
\|\nablah e_h\| = \|\nablah (u_h- \cI_hu)\|
\]
versus the number of elements with a logarithmic scaling
on both axes. We obtain the expected linear rate of 
convergence for $\g \ge 3/2$. The decay of the error
for $\g=3/2$ is different when $c_\a^{-1}=1$ instead of
$c_\a^{-1}=4$. 

\begin{figure}[h!]
\includegraphics[width=6cm]{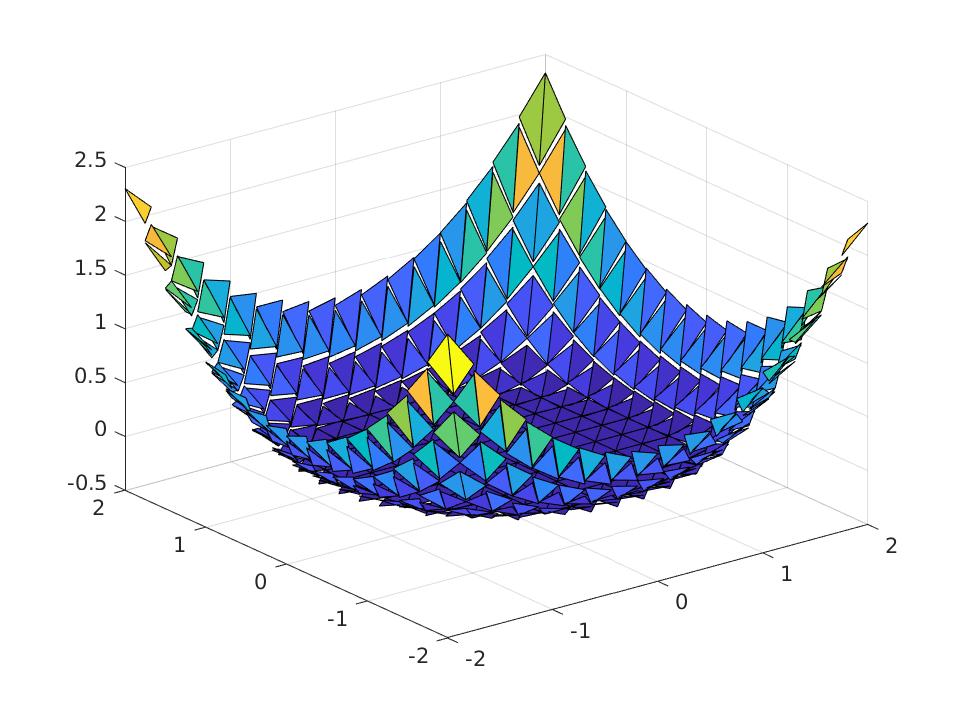}
\includegraphics[width=6cm]{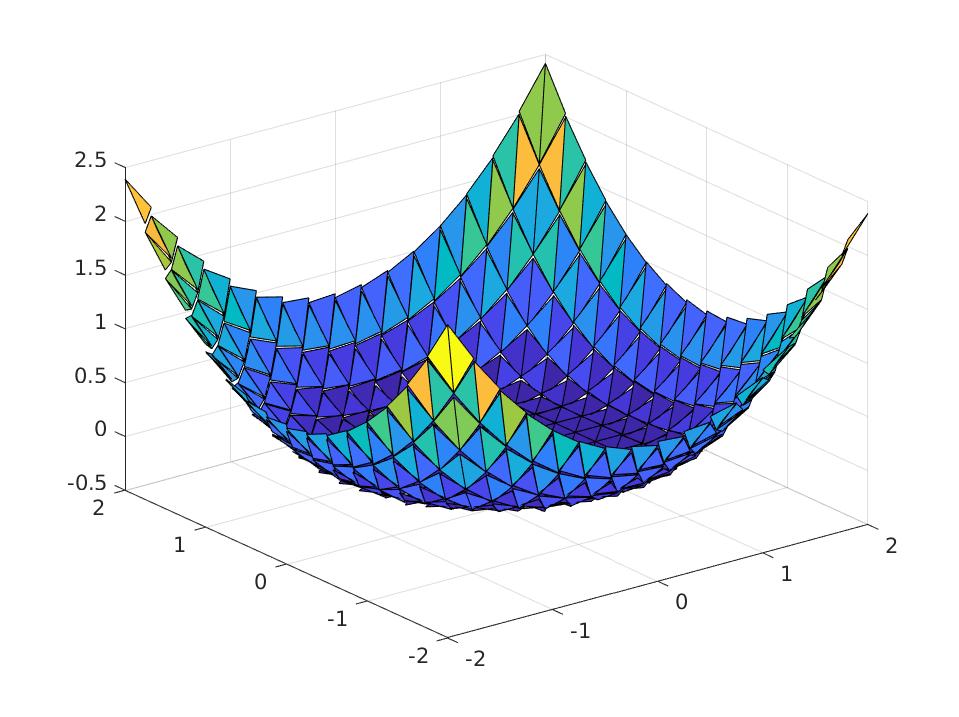}
\caption{\label{fig:obst_red_4} Discontinuous Galerkin solution
for the obstacle problem defined in Example~\ref{ex:obst}
with different penalty functionals.}
\end{figure}

\begin{figure}[h!]
\includegraphics[width=9.2cm]{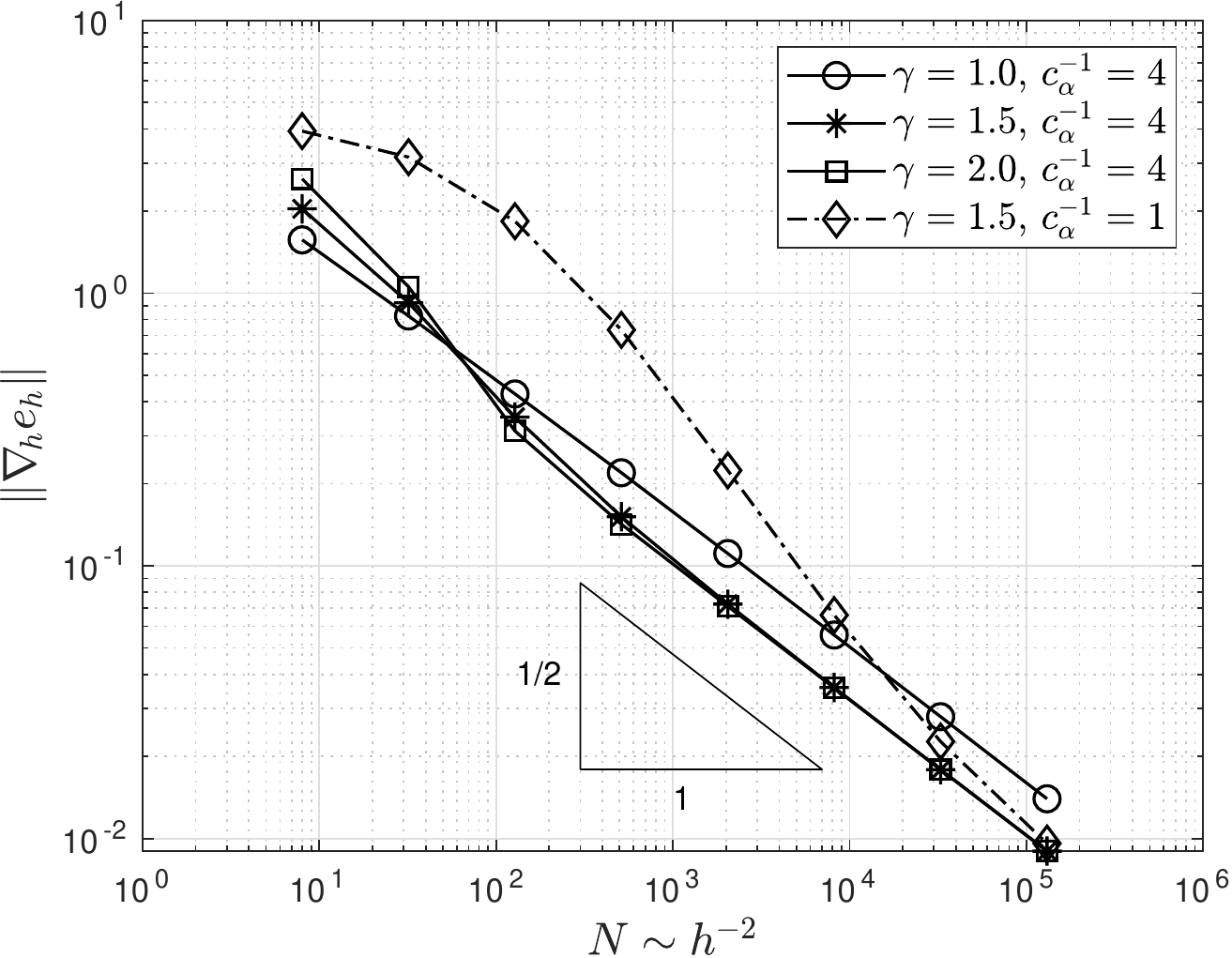}
\caption{\label{fig:obst_eoc} Experimental convergence rates
in the approximation of the obstacle problem defined in 
Example~\ref{ex:obst} for different penalty functionals.}
\end{figure}

\section*{References}
\printbibliography[heading=none]

\end{document}